\documentclass[12pt]{article}

\usepackage{amsmath} 
\usepackage{amsfonts} 
\usepackage{amssymb} 
\usepackage{amsthm}
\usepackage{graphicx}
\usepackage{latexsym}
\usepackage{amsmath,bm}
\usepackage{wrapfig}
\usepackage{enumerate}
\usepackage{mathtools}
\usepackage[nottoc,numbib]{tocbibind}
\usepackage{verbatim}
\usepackage[toc,page]{appendix}
\usepackage{hyperref}
\usepackage{titlefoot}

\newtheorem{theorem}{Theorem}
\theoremstyle{plain}

\newtheorem{claim}{Claim}

\newtheorem{definition}{Definition}

\newtheorem{lemma}{Lemma}

\newtheorem{proposition}{Proposition}
\newtheorem{corollary}{Corollary}
\newtheorem{remark}{Remark}

\numberwithin{equation}{section}
\numberwithin{theorem}{section}
\numberwithin{proposition}{section}
\numberwithin{lemma}{section}
\numberwithin{cor}{section}
\numberwithin{definition}{section}
\numberwithin{remark}{section}
\numberwithin{claim}{section}

\newcommand{\ds}{ \displaystyle }

\newcommand{\R}{\mathbb R}
\newcommand{\C}{\mathbb C}

\newcommand{\la}{\lambda}
\newcommand{\e}{\epsilon}

\begin{document}

\title
{Global, decaying solutions of a focusing energy-critical heat equation in $\mathbb{R}^4$}
\author{ Stephen Gustafson, Dimitrios Roxanas   \\ \\
Department of Mathematics\\
University of British Columbia\\
V6T 1Z2 Vancouver, Canada\\
gustaf@math.ubc.ca, droxanas@math.ubc.ca }
\date{}
\maketitle

\unmarkedfntext{\textbf{Keywords:} Nonlinear Heat Equation, Concentration Compactness, Regularity, Asymptotic Decay. \textbf{2010 AMS Mathematics Classification:} 35K05, 35B40, 35B65}

\setcounter{page}{1}

\begin{abstract} We study solutions of the focusing energy-critical nonlinear heat equation
$u_t = \Delta u - |u|^2u$ in $\mathbb{R}^4.$ We show that solutions emanating from initial data with 
energy and $\dot{H}^1-$norm below those of the stationary solution $W$ are global and decay to zero, via the 
``concentration-compactness plus rigidity'' strategy of Kenig-Merle \cite{KMa, KMb}.
First, global such solutions are shown to dissipate to zero, using a
refinement of the small data theory and the $L^2$-dissipation relation. 
Finite-time blow-up is then ruled out using the backwards-uniqueness
of Escauriaza-Seregin-Sverak \cite{ESS1,ESS2}
in an argument similar to that of Kenig-Koch \cite{KK} for the Navier-Stokes equations.
\end{abstract}

\tableofcontents

\section{Introduction}

We consider here the Cauchy problem for the 
focusing, energy-critical nonlinear heat equation in four space dimensions:
\begin{equation} \label{CP}
  \left\{ \begin{array}{l} 
  u_t = \Delta u + |u|^2u \\
  u(0,x) = u_0(x) \in \dot{H}^1(\R^4)
  \end{array} \right. 
\end{equation}
for $u(x,t) \in \C$ with initial data in the energy space 
\[
  \dot H^1(\R^4) = \{ u \in L^4(\R^4;\C) \; \big| \;  \| u \|_{\dot H^1}^2 = \int_{\R^4} |\nabla u(x)|^2 \; dx < \infty \}. 
\]
This is the $L^2$ gradient-flow equation for an {\it energy},  defined for $u \in \dot{H}^1$ as 
\[
  E(u) = \int_{\mathbb{R}^4} \left( \frac{1}{2} | \nabla u|^2 - \frac{1}{4} |u|^4 \right) dx,
\]  
and so in particular the energy is (formally) dissipated along solutions of~\eqref{CP}:
\begin{equation} \label{dissipation} 
  \frac{d}{dt} E(u(t)) = - \int_{\mathbb{R}^4} |u_t|^2 \; dx \leq 0.     
\end{equation}
We refer to the gradient term in $E$ as the {\it kinetic energy}, and the second term as the 
{\it potential energy}. The fact that the potential energy is negative expresses the focusing nature of the nonlinearity.  Problem~\eqref{CP} is {\it energy-critical} in the sense that the scaling 
\begin{equation}  \label{scaling}
  u_{\lambda}(t,x) = \lambda u(\lambda^2 t, \lambda x), \quad \la > 0
\end{equation}
leaves invariant the equation, the potential energy, and in particular the kinetic energy, which is the 
square of the {\it energy norm} $\| \cdot \|_{\dot H^1}$.\\

Static solutions of~\eqref{CP}, which play a key role here, solve the elliptic equation
\begin{equation} \label{static}
  \Delta W + |W|^2 W = 0.
\end{equation}
The function
\begin{equation*} 
  W = W(x)  = \frac{1}{(1+ \frac{|x|^2}{8})} \; \in \dot H^1(\R^4), \;\; \not\in L^2(\R^4)
\end{equation*} 
is a well-known solution. Its scalings by~\eqref{scaling}, and spatial translations of these
are again static solutions, and multiples of these
are well-known~\cite{A,T} to be the unique extremizers of the Sobolev inequality
\begin{equation}  \label{Sobolev} 
  \forall u \in  \dot{H}^1, \; \|u\|_{L^4} \leq C \|\nabla u \|_{L^2}, \quad 
  C = \frac{\|W\|_{L^4}}{\|\nabla W\|_{L^2}} = \|\nabla W\|_{L^2}^{-\frac{1}{2}}
  \mbox{ the best constant.}
\end{equation}

As for time-dependent solutions, a suitable local existence theory -- see Theorem~\ref{lwp} for details -- ensures the existence of a unique smooth solution $u \in C(I ; \dot H^1(\R^4))$ on a maximal time 
interval $I = [0,T_{max}(u_0))$. 
The main result of this paper states that
initial data lying ``below'' $W$ gives rise to global smooth solutions of~\eqref{CP}
which decay to zero:  
\begin{theorem} \label{Main Theorem} 
Let $u_0 \in \dot{H}^1(\R^4)$ satisfy
\begin{equation} \label{conditions}
  E(u_0) \leq E(W), \qquad \| \nabla u_0 \|_{L^2}  < \| \nabla W \|_{L^2}.
\end{equation}
Then the solution $u$ of~\eqref{CP} is global ($T_{max}(u_0) = \infty$) and satisfies
\begin{equation} \label{decays}
  \lim_{t \rightarrow \infty} \|u(t)\|_{\dot{H}^1} = 0. 
\end{equation}
\end{theorem}
The conditions~\eqref{conditions} define a non-empty set, since 
by the Sobolev inequality~\eqref{Sobolev} it includes all initial data of sufficiently small kinetic energy. 
Moreover, conditions~\eqref{conditions} are sharp for global existence and decay in several senses.
Firstly, if the kinetic energy inequality is replaced by equality, $W$ itself provides a non-decaying
(though still global) solution. 
Secondly, if the kinetic energy inequality is reversed, and under the additional assumption
$u_0 \in L^2(\R^4)$, by a slight variant of a classical argument~\cite{Lev} we find that   
the solution blows up in finite time:
\begin{theorem}  \label{blowup crit1} 
Let $u_0 \in H^1(\mathbb{R}^4)$ with
\[
  E(u_0) < E(W), \qquad \|\nabla u_0\|_{L^2} \geq \|\nabla W\|_{L^2}.
\]
Then the solution $u$ of~\eqref{CP} has finite maximal lifespan: $T_{max}(u_0) < \infty$.
\end{theorem}
Thirdly, for any $a^* > 0$, \cite{Sc} constructed finite-time blow-up solutions with initial data 
$u_0 \in H^1(\mathbb{R}^4)$ satisfying $E(W) < E(u_0) < E(W) + a^*$.
See also \cite{FHV} for formal constructions of blow-up solutions close to $W$.

It follows from classical variational bounds -- see Lemma~\ref{VarLemma} -- 
and energy dissipation~\eqref{dissipation}, that any solution
$u$ on a time interval $I = [0,T)$ whose initial data satisfies~\eqref{conditions}, necessarily satisfies
\begin{equation}   \label{result finite}
  \sup_{t \in I} \|\nabla u \|_{L^2} < \|\nabla W\|_{L^2}.
\end{equation} 
So it will suffice to show that the conclusions of Theorem~\ref{Main Theorem} hold
for any solution satisfying~\eqref{result finite}.
Indeed, we will prove:
\begin{enumerate}
\item 
If $I = [0,\infty)$ and~\eqref{result finite} holds, then $\ds \lim_{t \to \infty} \| \nabla u(t) \|_{L^2} = 0$.
This is given as Theorem~\ref{Decay Theorem}.
\item 
For any solution satisfying~\eqref{result finite}, $T_{max}(u(0)) = \infty$.
This is given as Corollary~\ref{Global Theorem}.
\end{enumerate}

That static solutions provide the natural threshold for global existence and decay, 
as in~\eqref{result finite}, is a classical phenomenon (eg. \cite{St}) for critical equations, particularly well-studied in the setting of
parabolic problems, mostly on compact domains, (e.g., \cite{Du,Gibook,LW,Ta}) via ``blow-up''-type  arguments:
first, failure of a solution to extend smoothly is shown, by a local regularity estimate, to imply (kinetic) energy concentration; then, near a point of concentration, 
rescaled subsequences are shown to converge locally to a non-trivial static solution; 
finally, elliptic/variational considerations prohibit non-trivial static solutions below the threshold. 

The main purpose of our work is twofold: first, to establish the global-regularity-below-threshold 
result Theorem~\ref{Main Theorem} on the full space $\R^4$; second, to do so not by way of the classical strategy sketched above, but instead via Kenig-Merle's \cite{KMa,KMb} 
``concentration-compactness plus rigidity" approach to critical {\it dispersive} equations, similar to Kenig-Koch's \cite{KK} implementation for the Navier-Stokes equations.

The argument is structured as follows.
First, in Section~\ref{no infinite}, we prove the energy-norm decay of {\it global} solutions which satisfy~\eqref{result finite}, Theorem~\ref{Decay Theorem}.
The strategy is that employed for the Navier-Stokes equations in \cite{GIP}:
reduce the problem to establishing the decay of {\it small} solutions (which is 
a refinement of the local theory) by exploiting the $L^2-$dissipation relation,
using a solution-splitting argument to overcome the fact that the solution
fails to lie in $L^2$. 

Second, in Section~\ref{Minimal Section}, we prove the existence and compactness 
(modulo symmetries) of a ``critical'' element -- a counterexample to global existence and decay,
which is minimal with respect to $\ds \sup_t \| \nabla u(t) \|_{L^2}$, 
following closely the work \cite{KV0}. See Theorem~\ref{critical thm}.
The technical tools are a profile decomposition compatible
with the heat equation (described in Section~\ref{profile}) and a perturbation result for the linear 
heat equation, based on the local theory (Proposition~\ref{perturbation}).

Finally, in Section~\ref{Rigidity Section}, we exclude the possibility of a compact  
solution with finite maximal existence time in Theorem~\ref{Rigidity Theorem}.
In fact this is a much stronger conclusion than required for the proof 
of Theorem~\ref{Main Theorem}, since it excludes {\it compact} finite-time blowup
at {\it any} kinetic energy level -- that is, it does not use~\eqref{result finite}.
This part is based on classical parabolic tools. We first show that the centre of compactness
remains bounded, by exploiting energy dissipation. Then a local small-energy regularity 
criterion, together with backwards uniqueness and unique continuation theorems
of~\cite{ESS1,ESS2}, as in \cite{KK}, imply the triviality of the critical element.

There is a vast literature on the semilinear heat equation $u_t = \Delta u + |u|^{p-1} u$. 
We content ourselves here with a brief review focused on the case of domain $\mathbb{R}^d$, and refer the reader to the recent book \cite{QS} for a more comprehensive review of the literature. 
For treatments of the Cauchy problem in $L^p$ and Sobolev spaces under various assumptions on the nonlinearity and the initial data, see \cite{W,W1,BC}. 

Much of the work concerns (energy) subcritical ($p < \frac{d+2}{d-2}$) problems. 
The seminal papers \cite{GK1,GK2,GK3} introduced the study of heat equations through similarity variables and characterized
blow-up solutions. In continuation of these works, \cite{M} gave a first construction of a solution with arbitrarily given blow-up points, and see \cite{MZ} (and references therein) for estimates of the blow-up rate, descriptions of the blow-up set, and stability results for the blow-up profile. We remark that blow-up in the subcritical case for $L^{\infty}-$solutions is known to be of Type I, in the sense that 
$\ds \limsup_{t \rightarrow T_{\text{max}}} (T_{\text{max}}-t)^{\frac{1}{p-1}} \|u(\cdot,t)\|_{L^\infty} < +\infty$, 
and Type I blow-up solutions are known to behave like self-similar solutions near the blow-up point. 
For a different set of criteria for global existence/blow-up in terms of the initial data we refer the reader to \cite{CDW}.
For results on the relation between the regularity of the nonlinear term and the regularity of the corresponding solutions, see \cite{CDW1}.

For supercritical problems, \cite{MM1,MM2,MM3} show that there is no Type II blow-up for 
$3 \leq d \leq 10,$ while for $d \geq 11$ it is possible if $p$ is large enough.  
It is also shown that a Type I blow-up solution behaves like a self-similar solution, while a Type II converges (in some sense) to a stationary solution. We also refer to the recent results \cite{C} ($d\geq 11,$ bounded domain), and \cite{CRS} and to the preprint \cite{BS} for results in Morrey spaces.

For the critical case, we have already mentioned the finite-time blow-up constructions \cite{FHV,Sc},
and we point to recent constructions of infinite-time blowup (bubbling) on bounded domains 
($d\geq 5$)~\cite{CDM}, and on $\R^3$~\cite{MW}. 
The work \cite{GV} deals with the continuation problem for reaction-diffusion equations. 
We finally mention the recent result~\cite{CMR}, where a complete classification of solutions 
sufficiently close to the stationary solution $W$ is provided for $d \geq 7$:
such solutions either exhibit Type-I blow-up; dissipate to zero; or converge to (a slightly rescaled, translated) $W$. In particular, Type II blow-up is ruled out in $d\geq 7$ near $W.$ We also refer to our work \cite{GR0} for a critical case of the $m$-corotational Harmonic Map Heat Flow.

\begin{remark}
We expect Theorem~\ref{Main Theorem} to extend to the energy critical 
problem for the nonlinear heat equation in general dimension $d \geq 3$:
\begin{equation} 
\begin{split}
  &u_t = \Delta u + |u|^{\frac{4}{d-2}}u \label{CP_d} \\
  &u(t_0,x) = u_0(x) \in \dot{H}^1(\mathbb{R}^d).
\end{split}
\end{equation}
For simplicity of presentation, we will give the proof only for the case $d=4$. 
As will be apparent from the proof, the result can be easily transferred to solutions of~\eqref{CP_d} for $d=3$. The proof should also carry over to $d \geq 5$ 
with some extra work to estimate the low-power nonlinearity as in~\cite{TV}.
 \end{remark}

\begin{remark}
Our proof makes no use of any parabolic comparison principles, and so applies to 
complex-valued solutions. That said, for ease of writing some estimates we 
will sometimes replace the nonlinearity $|u|^2 u$ with $u^3$ though the estimates
remain true in the $\C$-valued case.
\end{remark}

\section{Some analytical ingredients}

\subsection{Local theory}

We first make precise what we mean by a {\it solution} in the energy space:  
\begin{definition}
A function $ u: I \times \mathbb{R}^4 \to \C$ on a time interval $I = [0,T)$
($0 < T \leq \infty$) is a solution of (\ref{CP})
if $u \in (C_t \dot{H}^1_x \cap L^{6}_{t,x})([0,t] \times \mathbb{R}^4)$;
$\nabla u \in L^3_{x,t} ([0,t] \times \mathbb{R}^4)$;
$D^2 u, \; u_t \in L^2_t L^2_x([0,t] \times \R^4)$ for all $t \in I$; 
and the Duhamel formula
\begin{equation}  \label{duhamel}
  u(t) =  e^{t \Delta} u_0 + \int_{0}^t e^{(t-s)\Delta} F(u(s)) ds ,
\end{equation}
is satisfied for all $t \in I$, where $F(u) = |u|^{2}u.$ 
We refer to the interval I as the lifespan of u. We say that u is a maximal-lifespan solution if the solution cannot be extended to any strictly larger interval. We say that u is a global solution if $I = \mathbb{R}^{+}:=[0,+\infty).$
\end{definition}
We will often measure the space-time size of solutions on a time interval $I$ in 
$L^6_{x,t}$,  denoting
\[
  S_I (u) := \int_I \int_{\mathbb{R}^4} |u(t,x)|^{6} dx dt, \quad
  \|u\|_{S(I)} := S_I(u)^{\frac{1}{6}} = \left(\int_I \int_{\mathbb{R}^4} |u(t,x)|^6 dx dt \right)^{\frac{1}{6}}.
\]

A local wellposedness theory in the energy space $\dot H^1(\R^4)$, 
analogous to that for the corresponding critical 
nonlinear Schr\"odinger equation (see e.g. \cite{CW}), is easily constructed, based on the Sobolev inequality and space-time estimates  for the heat equation on $\R^4$ (\cite{Gi}),
\begin{equation}  \label{STdecayheat}
\begin{split} 
  & \| e^{t \Delta} \phi \|_{L^p_x(\mathbb{R}^4)} \lesssim t^{-2(1/a - 1/p)} \| \phi \|_{L^a}, 
  \qquad 1 \leq a \leq p \leq \infty  \\
  & \| e^{t \Delta} \phi\|_{L^q_t L^p_x(\mathbb{R}^+ \times \mathbb{R}^4)} 
  \lesssim \| \phi \|_{L^a}, \qquad  \frac{1}{q} + \frac{2}{p} = \frac{2}{a}, \quad 1 < a \leq q  \\ 
  & \| \int_0^t e^{(t-s) \Delta} f(s) ds \|_{L^q_t L^p_x(\mathbb{R}^+ \times \mathbb{R}^4)} \lesssim  
  \| f \|_{L^{\tilde{q}'}_t L^{\tilde{p}'}_x(\mathbb{R}^+ \times \mathbb{R}^4)},   \\
  &\qquad \qquad \frac{1}{q} + \frac{2}{p} =   \frac{1}{\tilde{q}} + \frac{2}{\tilde{p}} = 1,
  \qquad \frac{1}{q} + \frac{1}{q'} = \frac{1}{p} + \frac{1}{p'} = 1,\end{split}
\end{equation}$ \hspace{0.3em} \text{and} \hspace{0.3em} (\tilde{q}',\tilde{p}') \hspace{0.3em} \text{the dual to any admissible pair} \hspace{0.3em} (\tilde{q},\tilde{p}).$

We also refer the reader to \cite{BC0, W} for a treatment of the Cauchy problem in the critical 
Lebesque space $L^{\frac{2d}{d-2}}$; the arguments directly adapt to show wellposedness in $\dot{H}^1$. 
One can use a fixed-point argument to construct local-in-time solutions for arbitrary initial data in 
$\dot{H}^1(\mathbb{R}^4)$; however, as usual when working in critical scaling spaces, 
the time of existence depends on the profile of the initial data, not merely on its $\dot{H}^1$-norm. 
We summarize:
\begin{theorem} \label{lwp} (Local well-posedness) 
Assume $u_0 \in \dot{H}^1(\mathbb{R}^4).$  
\begin{enumerate} 
\item (Local existence) 
There exists a unique, maximal-lifespan solution to the Cauchy Problem~\eqref{CP} in 
$I \times \mathbb{R}^4, I = [0, T_{max}(u_0))$
\item(Continuous dependence) 
The solution depends continuously on the initial data (in both the $\dot{H}^1$ and the $S_I$-induced topologies). 
Furthermore, $T_{\text{max}}$ is a lower-semicontinuous function of the initial data.
\item(Blow-up criterion) 
If \hspace{0.1em} $T_{\text{max}}(u_0) < + \infty,$ then $ \| u \|_{S([0,T_{\text{max}}(u_0)))} = + \infty$ 
\item(Energy dissipation) the energy E(u(t)) is a non-increasing function in time. 
More precisely,
for $0 < t < T_{max}$,
\begin{equation}   \label{energy dissipation}
  E(u(t)) + \int_0^t \int_{\R^4} |u_t|^2 \; dx \; dt = E(u_0).
\end{equation}
\item(Small data global existence)
There is $\epsilon_0>0$ such that if $\|e^{t\Delta} u_0\|_{S(\R^+)} \leq \epsilon_0$,
the solution $u$ is global, $T_{max}(u_0) = \infty$, and moreover 
\begin{equation}  \label{small data}
  \| u \|_{S(\R^+)} + \| \nabla u \|_{(L^\infty_t L^2_x \cap L^3_{x,t})(\R^+\times \R^4)}
 + \| D^2 u \|_{L^2_{x,t}(\R^+\times \R^4)} \lesssim \e_0.
\end{equation} 
This occurs in particular when $ \|u_0\|_{\dot{H}^1(\mathbb{R}^4)}$ is sufficiently small.  
\end{enumerate}
\end{theorem}

An extension of the proof of the local existence theorem implies the following stability result (see, e.g., \cite{KV}):
\begin{proposition} \label{perturbation} (Perturbation result)\\
  For every $E, L > 0$ and $\epsilon > 0$ there exists $\delta > 0 $ with the following property: assume $\tilde{u}: I \times \mathbb{R}^4 \to \mathbb{R},$ $I = [0,T),$ is an approximate solution to (\ref{CP}) in the sense that 
\begin{equation*}  \label{error} 
  \| \nabla e \|_{L^{\frac{3}{2}}_{t,x} (I \times \mathbb{R}^4)} \leq \delta, \quad
  e:= \tilde{u}_t - \Delta \tilde{u}- |\tilde{u}|^2\tilde{u},
\end{equation*}
and also
\begin{equation*}
\|\tilde{u}\|_{L^{\infty}_t \dot{H}^1_x(I \times \mathbb{R}^4)} \leq E \hspace{1em} \text{and} \hspace{1em} \| \tilde{u} \|_{S(I)} \leq L,
\end{equation*}
then if $u_0 \in \dot{H}^1_x (\mathbb{R}^4)$ is such that
\begin{equation*} 
  \|u_0 - \tilde{u}(0) \|_{\dot{H}^1_x(\mathbb{R}^4)} \leq \delta, \end{equation*}
there exists a solution $ u: I \times \mathbb{R}^4 \to \mathbb{R}$ of~\eqref{CP} with $u(0) = u_0$, and such that
\begin{equation*} 
  \|u - \tilde{u} \|_{L^{\infty}_t \dot{H}^1_x(I \times \mathbb{R}^4)} +\| u-\tilde{u} \|_{S(I)} \leq \epsilon .\end{equation*}
\end{proposition} 

\subsection{Profile decomposition}  \label{profile}

The following proposition is the main tool (along with the Perturbation Proposition~\ref{perturbation}) used to establish the existence of a critical element.
The idea is to characterize the loss of compactness in some critical embedding;
it can be traced back to ideas in \cite{L}, \cite{BC}, \cite{St1}, \cite{SU} and their modern ``evolution'' counterparts  \cite{BG}, \cite{KMa} and \cite{KMb}.
 
\begin{proposition} \label{PD} (Profile Decomposition)
Let $\{u_n\}_n$ be a bounded sequence of functions in $\dot{H}^1 (\mathbb{R}^4)$. 
Then, after possibly passing to a subsequence (in which case, we rename it $u_n$), there exists a family of functions 
$ \{ \phi^j \}^{\infty}_{j=1} \subset \dot{H}^1,$ scales $\lambda^j_n > 0$ and centers $x^j_n \in \mathbb{R}^4 $ such that: 
\[
  u_n (x) = \sum_{j=1}^J \frac{1}{\lambda^j_n} \phi^j( \frac{x-x^j_n}{\lambda^j_n}) +  w^J_n(x), 
\]
$w^J_n \in \dot{H}^1(\mathbb{R}^4) $ is such that:
\begin{equation} 
  \lim_{J \rightarrow \infty} \limsup_n \| e^{t\Delta}  w^J_n \|_{L^6_{t ,x}(\mathbb{R}^+ \times \mathbb{R}^4)} = 0,  
\label{holygrailH} 
\end{equation}
\begin{equation}  
  \lambda^j_n  w^J_n(\lambda^j_n x + x^j_n) \rightharpoonup 0, \hspace{0.6em} \text{in} \hspace{0.6em} 
  \dot{H}^1   (\mathbb{R}^4), \hspace{0.6em}  \forall j \leq J. \label{wwconv} 
\end{equation} 
Moreover, the scales are asymptotically orthogonal, in the sense that 
\begin{equation} 
  \frac{\lambda^j_n}{\lambda^{i}_n} + \frac{\lambda^{i}_n}{\lambda^j_n}+ \frac{|x^i_n - x^j_n|^2}{\lambda^{j}_n \lambda^{i}_n} \rightarrow + \infty, \hspace{0.6em} \forall i \neq j \label{orth}. 
\end{equation}
Furthermore, $ \text{for all} \hspace{0.6em} J \geq 1$ we have the following decoupling properties:
\begin{equation} 
  \|u_n\|^2_{\dot{H}^1} = \sum_{j=1}^J \| \phi^j \|^2_{\dot{H}^1} + \| w^J_n \|^2_{\dot{H}^1} + o_n(1) \label{H1dec} 
\end{equation}
and
\begin{equation} 
  E(u_n)= \sum_{j=1}^J E(\phi^j ) + E(w^J_n) + o_n(1). \label{Energy Splitting} 
\end{equation}
\end{proposition}
For details about the proof see \cite{thesis, GR0}.
\subsection{Variational estimates}

The elementary variational inequalities we use are summarized here:
\begin{lemma} \label{VarLemma} (Variational Estimates)
\begin{enumerate} 
\item
If
\[
  \| \nabla u_0 \|^2_{L^2} \leq \| \nabla W \|^2_{L^2}, \quad
  E(u_0) \leq (1-{\delta_0}) E(W), \;\; \delta_0 > 0 ,
\]  
then there exists $\bar{\delta} = \bar{\delta} (\delta_0) > 0$ such that
for all $t \in [0, T_{max}(u_0))$, the solution of~\eqref{CP} satisfies
\begin{equation}   \label{gradient trapping} 
  \int |\nabla u(t)|^2 \leq (1-\bar{\delta}) \int  |\nabla W|^2.   
\end{equation}
\item
If~\eqref{gradient trapping} holds, then
\begin{equation} \label{used trapping}
  \int (|\nabla u(t) |^2 - |u(t)|^4 ) dx  \geq \bar{\delta} \int  |\nabla u(t) |^2 
\end{equation}
and moreover $E(u(t)) \geq 0$.
\end{enumerate}
\end{lemma}
\begin{proof}
The second statements are an immediate consequence of the 
sharp Sobolev inequality~\eqref{Sobolev}:
\[
  \int (|\nabla u(t) |^2 - |u(t)|^4 ) dx  \geq \left[1-\left(\frac{\| \nabla u(t) \|_{L^2}}{\| \nabla W \|_{L^2}}\right)^2
  \right] \| \nabla u  \|^2_{L^2} \gtrsim \| \nabla u \|_{L^2}^2
\] 
while the first follows easily from Sobolev and energy dissipation~\eqref{energy dissipation};
see, e.g., Lemma 3.4/Theorem 3.9  in \cite{KMa}.
\end{proof}

\section{Asymptotic decay of global solutions} \label{no infinite}

In this section we prove the following theorem:
\begin{theorem} \label{Decay Theorem} 
If $u \in C([0,\infty); \dot{H}^1(\mathbb{R}^4))$ is a solution to equation (\ref{CP}) which moreover satisfies 
\begin{equation}  \label{condition} 
  \sup_{t \geq 0} \| \nabla u(t) \|_{L^2} < \| \nabla W \|_{L^2}, 
\end{equation}   
then 
\begin{equation*} 
  S_{\R_+}(u) < \infty \;\; \mbox{ and } \;\;  
  \lim_{t \rightarrow \infty} \| u(t) \|_{ \dot{H}^1} = 0.
\end{equation*}  
\end{theorem}
\begin{proof}
The general strategy, drawn from the techniques of \cite{GIP} for the Navier-Stokes 
equations, is as follows. 
We first show that global solutions for which $S_{\R^+}(u) < \infty$
-- which includes {\it small} solutions by the small data theory~\eqref{small data} -- 
decay to zero in the $\dot{H}^1-$norm. 
Second, we impose the extra assumption of $H^1-$ data, so that we may 
exploit the $L^2-$dissipation relation to show finiteness of 
$\| \nabla u \|_{L^2_{x,t}}$, which in turns allows us to reduce matters 
to the case of small $\dot{H}^1$ data.
Finally, to remove this extra assumption, we split the initial data in frequency, and estimate a 
perturbed equation. 

\begin{proposition} \label{decay global}
If $u$ is a global solution of~\eqref{CP} with $S_{\R^+}(u) < \infty$, then
\begin{equation} \label{decayto0}
   \lim_{t \to \infty} \| \nabla u (t,\cdot) \|_{L^2} = 0.
\end{equation}
\end{proposition}
\begin{proof}
Let  $u \in (C_t \dot{H}^1_x \cap L^{6}_{t,x}) (\mathbb{R}_{+} \times \mathbb{R}^4)$
be a global solution to~\eqref{CP}.
Just as one proves the blow-up criterion for the local theory Theorem~\ref{lwp},
we first show:
\begin{claim} 
$\| \nabla u\|_{L^{3}_{t,x}(\R^+ \times \R^4)} < \infty.$ 
\end{claim}
\begin{proof}
Since $u \in L^6_{t,x}(\R^+ \times \R^4)$, given $\eta > 0$, we may subdivide 
$R^+ = [0,\infty)$ into a finite number of subintervals 
$I_j = [a_j, a_{j+1})$, $j = 0,1,\ldots, J$, $0 = a_0 < a_1 < \cdots < a_J = \infty$,
on which $\|u\|_{L^{6}_{t,x}} (I_j) \leq \eta$.  
Taking $\nabla$ in the Duhamel formula~\eqref{duhamel} and using~\eqref{STdecayheat}:
\begin{equation*} 
\begin{split}
  \| \nabla u \|_{(L^{3}_{t,x} \cap L^\infty_t L^2_x)}(I_0) & \leq C \|e^{t\Delta} \nabla u_0\|_{L^2} 
  + C\| \int_0^t S(t-s) \nabla(u^3) ds \|_{L^{3}_{t,x}(I_0)}  \\
  &\leq C\|u_0\|_{\dot{H}^1} + C \| u^2 \nabla u\|_{L^{3/2}_{t,x}(I_0)} \\
  & \leq C\|u_0\|_{\dot{H}^1} + C \|u\|^2_{L^{6}_{t,x}(I_0)} \| \nabla u\|_{L^{3}_{t,x}(I_0)},
\end{split}
\end{equation*}
so by choosing $\eta < \frac{1}{\sqrt{2C}}$ we ensure
\begin{equation*} 
  \| \nabla u\|_{(L^{3}_{t,x} \cap L^\infty_t L^2_x)(I_0)} \leq 2 C \|u_0\|_{\dot{H}^1}. 
\end{equation*}
In particular $\| u(a_1)\|_{\dot H^1} \leq 2C \| u_0 \|_{\dot H^1}$, and so 
we may repeat this argument on the next interval $I_1$ to find
$\| \nabla u\|_{(L^{3}_{t,x} \cap L^\infty_t L^2_x)(I_1)} \leq (2 C)^2 \|u_0\|_{\dot{H}^1}$,
and, continuing, 
$\| \nabla u\|_{(L^{3}_{t,x} \cap L^\infty_t L^2_x)(I_j)} \leq (2 C)^{j+1} \|u_0\|_{\dot{H}^1}$,
for $j = 0,1,\ldots,J$. The claim follows.
\end{proof}
Now denote the linear evolution by $S(t) = e^{t \Delta}$, so the solution in Duhamel form 
is written
\[
  u(t) = S(t) u_0 + \int_{0}^t S(t-s) u^3(s) ds. 
\]
Let \begin{equation*}
  \text{I} := S(t) u_0, \quad \text{II} := \int_{0}^{\tau} S(t-s) u^3(s) ds, \quad
  \text{III}:= \int_{\tau}^t S(t-s) u^3(s) ds,
\end{equation*}
for some $\tau$ to be determined later.\\

For term $\text{I}$ we will take advantage of the decay of the heat propagator. 
By density, we can approximate $\nabla u_0$ by $v \in L^1 \cap L^2$  
and use a standard heat estimate:
\begin{equation*}
\begin{split}
  \|\text{I}\|_{\dot{H}^1} = \|S(t) \nabla u_0 \|_{L^2} 
  &\leq \|S(t) (\nabla u_0 - v) \|_{L^2}  + \|S(t) v \|_{L^2} \\ 
  & \leq  \|\nabla u_0 - v \|_{L^2} + \|S(t) v \|_{L^2}. 
\end{split}
\end{equation*}

The first term can be made arbitrary small by the choice of $v$, while for the second,
by (\ref{STdecayheat}), $\|S(t) v \|_{L^2} \rightarrow 0$ as $ t \rightarrow \infty$, hence 
\[
  \| I \|_{\dot H^1} \to 0 \mbox{ as } t \to \infty.
\]

We now treat term III, which will allow us to fix $\tau.$ 
By the claim, for any $ \epsilon > 0, $ we can find $\tau$ such that $\|u\|_{L^{6}_{t,x}([\tau,\infty) \times \mathbb{R}^4)}, \| \nabla u\|_{L^{3}_{t,x}([\tau,\infty) \times \mathbb{R}^4)} \leq \epsilon.$ 
Since we are considering the limit $t \rightarrow \infty$, we may assume 
$ t > \tau \gg 1$, and so by the same estimate of the nonlinear term as in the proof
of the claim, 
\begin{equation*} 
  \|III \|_{\dot{H}^1} \lesssim  \|u\|_{L^{6}_{t,x}([\tau,t) \times \mathbb{R}^4)}^2
  \| \nabla u\|_{L^{3}_{t,x}([\tau,t) \times \mathbb{R}^4)} \lesssim \epsilon^3.
\end{equation*} 

  Having fixed $\tau$ in this manner, we turn to term II.  First notice that
\begin{equation*} 
  II = \int_0^{\tau} S(t-s) u^3(s) ds =  S(t-\tau)  \int_0^{\tau} S(\tau-s) u^3(s) ds.  
\end{equation*}
Since $\ds \int_0^{\tau} S(\tau-s) u^3(s) ds  \in \dot{H}^1$ (by $u \in L^6_{x,t}$ and~\eqref{STdecayheat}),
the same approximation argument used for term I shows 
\begin{equation*} 
  \| II \|_{\dot H^1} = 
  \| S(t-\tau)  \int_0^{\tau} S(\tau-s) u^3(s) ds \|_{\dot{H}^1} \xrightarrow {t \to \infty} 0. \end{equation*}
Since $\e$ was arbitrary,~\eqref{decayto0} follows.
\end{proof}

Now if we assume $u_0 \in H^1(\R^4)$, multiplying~\eqref{CP} by $u$ and integrating
over space-time yields the $L^2$ dissipation relation 
\begin{equation} \label{L2disd}
  \| u(t) \|^2_{L^2} = \| u_0 \|^2_{L^2} + 2 \int_0^t \int_{\mathbb{R}^4} [ u^4 - |\nabla u|^2] dx ds .  
\end{equation} 
Because of~\eqref{condition}, we have the variational estimate (\ref{used trapping}) and so
for some $\bar{\delta} > 0$,
\begin{equation*} 
   \sup_{t \geq 0} \| u(t) \|^2_{L^2} + 2 \bar{\delta} \| \nabla u\|^2_{L^2_{t,x}( \mathbb{R}_{+} \times \mathbb{R}^4)}    \leq  \| u_0 \|^2_{L^2}.  
\end{equation*}
This estimate immediately implies that for any $\epsilon_0 > 0$, 
there is some time $t_0$ such that $\|u(t_0)\|_{\dot{H}^1} \leq \epsilon_0$, 
and we can directly apply the small data result~\eqref{small data} 
(with initial time $t = t_0$) to conclude that $S_{\R^+}(u) < \infty$, 
and so by Proposition~\ref{decay global},
$\ds \lim_{t \to \infty} \| u(t) \|_{\dot H^1} = 0$, as required.

To remove the extra assumption $u_0 \in L^2$, split
\[ 
  u_0 = w_0 + v_0, \qquad \|w_0\|_{\dot{H}^1} \ll 1,
  \qquad v_0 \in H^1.
\]
Define $w(t)$ to be the solution to~\eqref{CP} with initial data $w_0$:
\[
\begin{split}
&w_t = \Delta w + w^3 \\
&w(0,x) = w_0(x) \in \dot{H}^1(\mathbb{R}^4).
\end{split}
\]
From the small data theory~\eqref{small data}, 
$w \in C_t \dot H^1_x(\R_+ \times \R^4)$ is global, with
\begin{equation} \label{w small}
  \| w \|_{L^6_{t,x}(\mathbb{R}_{+} \times \mathbb{R}^4)}
  + \| \nabla w \|_{(L^\infty_t L^2_x \cap L^3_{t,x})(\mathbb{R}_{+} \times \mathbb{R}^4)}
  \lesssim \| \nabla w_0 \|_{L^2} \ll1
\end{equation}
and by Proposition~\ref{decay global}, 
$\ds \|w(t)\|_{\dot{H}^1} \xrightarrow{ t \to \infty} 0.$

Defining $v$ by $v:= u  - w,$ it will be a solution of the perturbed equation 
\begin{equation*} 
  v_t - \Delta v = v^3 + 3 w^2 v + 3 w v^2 .
\end{equation*}
Just as in the derivation of the $L^2$-dissipation relation~\eqref{L2disd}, 
multiply by $v$ and integrate in space-time: 
\begin{equation*} 
  \|v(t)\|^2_{L^2} - \|v_0\|^2_{L^2} +2 \int_0^t \|\nabla v\|^2_{L^2} = 2\int_0^t \|v\|^4_{L^4} 
  + 6 \int_0^t \int_{\mathbb{R}^4} w^2v^2 + 6 \int_0^t \int_{\mathbb{R}^4} w v^3 .
\end{equation*}
By~\eqref{w small}, picking $\| \nabla w_0 \|_{L^2}$ small enough, ensures that
condition~\eqref{condition} holds also for $v$: 
$\ds \sup_{t \geq 0} \| \nabla v(t) \|_{L^2} < \| \nabla W \|_{L^2}$. Hence by~\eqref{used trapping},
for some $\bar{\delta} > 0$,
\[
  \| v(t) \|^2_{L^2}  + \bar{\delta} \int_0^t \|\nabla v\|^2_{L^2} \lesssim
  \| v_0 \|^2_{L^2}  + \int_0^t \int_{\mathbb{R}^4} w^2v^2 + 6 \int_0^t \int_{\mathbb{R}^4} w v^3,
\]
and so by H\"older and Sobolev,
\[ 
\begin{split}
  \| v(t) \|^2_{L^2}  + \bar{\delta} \|\nabla v\|^2_{L^2 L^2} & \lesssim
  \| v_0 \|^2_{L^2}  + \| w \|^2_{L^{\infty} L^4} \| v \|_{L^2 L^4}^2
  + \| w \|_{L^\infty L^4} \| v \|_{L^\infty L^4} \| v \|_{L^2 L^4}  \\
  & \lesssim \| v_0 \|^2_{L^2}  + \| \nabla w \|^2_{L^{\infty} L^2} \| \nabla v \|_{L^2 L^2}^2\\
  &+ \| \nabla w \|_{L^\infty L^2} \| \nabla v \|_{L^\infty L^2} \| \nabla v \|_{L^2 L^2} .
\end{split}
\]
So by~\eqref{w small}, choosing $\| \nabla w_0 \|_{L^2}$ small enough yields 
$\ds \int_0^\infty \| \nabla v \|^2_{L^2} dt < \infty$, and hence there is $T > 0$ 
for which $ \| \nabla v(T) \|_{L^2} < \|w_0\|_{\dot{H}^1}$ and so  
$\|\nabla u(T)\|_{L^2} \leq 2 \|\nabla w_0\|_{L^2}.$ 
Choosing $\|\nabla w_0\|_{L^2}$ smaller still, if necessary,  we are able to apply the 
small data result~\eqref{small data} to conclude $S_{\R_+}(u) < \infty$, and 
moreover by Proposition~\ref{decay global},  
\begin{equation*}  
  \lim_{t \to \infty} \|u(t)\|_{\dot{H}^1} = 0, 
\end{equation*}  
concluding the proof of the theorem.
\end{proof}

\section{Minimal blow-up solution} \label{Minimal Section}

For any $ 0 \leq E_0 \leq \| \nabla W \|_2^2,$ we define
\begin{equation*}
  L(E_0):= \sup \{ S_I(u) \; | \;  u \mbox{ a solution of~\eqref{CP} on } I \mbox{ with } 
  \sup_{t \in I } \| \nabla u(t) \|_2^2 \leq E_0 \} ,
\end{equation*}
where $I = [0,T)$ denotes the existence interval of the solution in question.
$L: [0, \| \nabla W\|_2^2] \rightarrow [0, \infty]$ is a continuous (this follows from Proposition~\ref{perturbation}), 
non-decreasing function with  $L(\| \nabla W \|_2^2) = \infty.$  
Moreover, from the small-data theory~\eqref{small data}, 
$$ L(E_0) \lesssim E_0^3 \hspace{0.4em} \text{for} \hspace{0.4em}  E_0 \leq \e_0.$$
Thus, there exists a unique critical kinetic energy $E_c \in (0, \| \nabla W\|_2^2]$ such that 
\[
  L(E_0) < \infty \mbox{ for } E_0 < E_c, \quad L(E_0) = \infty \mbox{ for } E_0 \geq E_c.
\]
In particular, if $ u: I \times \mathbb{R}^4 \rightarrow \mathbb{R}$  is a maximal-lifespan 
solution, then
\[
  \sup_{t \in I} \| \nabla u(t) \|_2^2 < E_c \; \implies \;  u \mbox{ is global, and }
  \| u \|_{S(\mathbb{R}^{+})} \leq L(\sup_{t \in I} \| \nabla u(t) \|_2^2 ) 
  < \infty.
\]
The goal of this section is the proof of  the following theorem: 
\begin{theorem} \label{critical thm} 
There is a maximal-lifespan solution $u_c : I \times \mathbb{R}^4 \rightarrow \mathbb{R}$ to (\ref{CP}) such that 
$\ds \sup_{t \in I} \| \nabla u_c(t) \|^2_{L^2} = E_c$,  $\|u_c\|_{S(I)} = +\infty$.
Moreover, there are $x(t) \in \mathbb{R}^4, \lambda(t) \in \mathbb{R}^+, $ such that 
\begin{equation} \label{precompact}
  K= \left\{ \frac{1}{\lambda(t)} u_c \left( t, \frac{x-x(t)}{\lambda(t)} \right) \; \big| \; t \in I  \right\}
\end{equation}
is precompact in $\dot{H}^1$.
\end{theorem}
For the proof of this theorem we closely follow the arguments in \cite{KV0}.
The extraction of this minimal blow-up solution (and its compactness up to scaling and translation) will be a consequence of the following proposition:
\begin{proposition} \label{Palais-Smale}
Let $u_n : I_n \times \mathbb{R}^4$ be a sequence of solutions to (\ref{CP}) such that 
\begin{equation} 
\limsup_n \sup_{t \in I_n} \| \nabla u_n \|_2^2 = E_c  \hspace{0.4em} \text{and} \hspace{0.4em} \lim_{n \rightarrow \infty} \|u_n\|_{S(I_n)} = + \infty. 
\label{bad sequence} \end{equation} 
where $I_n$ are of the form $[0, T_n)$. 
Denote the initial data by $u_n(x,0) = u_{n,0}(x)$.
Then the sequence $\{ u_{n,0} \}_n$ converges, modulo scaling and translations, 
in $\dot{H}^1$ (up to an extraction of a subsequence). 
\end{proposition}
\begin{proof}
The sequence $\{ u_{n,0} \}_n $ is bounded in $\dot{H}^1$ by \eqref{bad sequence} so applying the profile decomposition (up to a further subsequence) we get
$$ u_{n,0}(x) = \sum_{j=1}^J \frac{1}{\lambda_n^j} \phi^j (\frac{x-x_n^j}{\lambda_n^j}) + w_n^J(x) $$ with the properties listed in Proposition~\ref{PD}.

Define the nonlinear profiles $v^j: I^j \times \mathbb{R}^4 \rightarrow \mathbb{R}$,
$I^j = [0, T^j_{max})$, associated to $\phi^j$ by setting them to be the maximal-lifespan solutions of (\ref{CP}) with initial data 
$v^j(0) =  \phi^j.$ 
Also, for each $j, n \geq 1$ we introduce 
$v_n^j: I_n^j \times \mathbb{R}^4 \rightarrow \mathbb{R}$  by 
\[
  v_n^j(t) =  \frac{1}{\lambda_n^j}  v^j \left( \frac{t}{(\lambda_n^j)^2}, \frac{x-x_n^j}{\lambda_n^j}
  \right), \quad
  I_n^j:= \{ t \in \mathbb{R}: \frac{t}{(\lambda_n^j)^2} \in I^j \} .
\]
Each $v_n^j$ is a solution with $v_n^j(0) =  \frac{1}{\lambda_n^j} \phi (\frac{x-x_n^j}{\lambda_n^j})$ and maximal lifespan $I_n^j  = [0, T^{n,j}_{\text{max}})$,
$T^{n,j}_{max} = (\lambda_n^j)^2 T^j_{max}$.

For large $n$, by the asymptotic decoupling of the kinetic energy (property (\ref{H1dec})), 
there is a $ J_0 \geq 1$ such that $\| \nabla \phi^j\|_2^2 \leq \e_0$ for all $ j \geq J_0$, 
where $\e_0$  is as in Theorem~\ref{lwp}, 4.
Hence, for $j \geq J_0$, the solutions $v_n^j$ are global and decaying to zero, and moreover 
\begin{equation} 
\sup_{t \in \mathbb{R}^{+}} \| \nabla v_n^j \|_2^2 + \| v_n^j\|^2_{S(\mathbb{R}^{+} \times \mathbb{R}^4)} \lesssim \| \nabla \phi^j \|_2^2 \label{when global} 
\end{equation} 
by the small data theory~\eqref{small data}.

\begin{claim}\label{claim1} (There is at least one bad profile).  
There exists $1 \leq j_0 < J_0 $ such that $ \| v^{j_0} \|_{S(I^{j_0})} = \infty$.
\end{claim}

\noindent For contradiction, assume that for all $1 \leq j < J_0$ 
\begin{equation}
   \| v^{j} \|_{S(I^{j})} < \infty  
\label{A} \end{equation}
which by the local theory implies $I^j = I^j_n = [0, \infty)$ for all such $j$ and for all $n$.
The goal is to deduce a bound on $ \|u_n\|_{S(I_n)} $for sufficiently large n. To do so, we will use Proposition~\ref{perturbation}, for which we first need to introduce a good approximate solution.

Define
\begin{equation} 
  u_n^J (t) = \sum_{j=1}^{J}  v_n^j(t) + e^{t \Delta} w_n^J .
\label{approximate}  
\end{equation}
We will show that for $n$ and $J$ large enough this is a good approximate solution (in the sense of Proposition \ref{perturbation}) and that $ \| u_n^J \|_{S([0,+\infty))}$ is uniformly bounded. The validity of both points implies that the true solutions $u_n$ should not satisfy \eqref{bad sequence}, reaching a contradiction.

First observe
\begin{equation} 
\sum_{j \geq 1}  \|v_n^j\|^2_{S([0,\infty))} = \sum_{j=1}^{J_0 - 1}  \|v_n^j\|^2_{S([0,\infty))}  +  \sum_{j \geq J_0}  \|v_n^j\|^2_{S([0,\infty))}  
\end{equation}
\begin{equation} 
\qquad \quad \lesssim 1 +  \sum_{j \geq J_0}  \| \nabla \phi^j\|_2^2 \lesssim 1 + E_c \end{equation}
where we have used \eqref{A}, property (\ref{H1dec})  and (\ref{bad sequence}).

Now, using the above and (\ref{holygrailH}) in Proposition~\ref{PD}:
\begin{equation}  
  \lim_{J \rightarrow \infty} \underset{n}{\overline{\text{lim} }} \hspace{0.2em} 
  \| u_n^{J} \|_{S([0,+\infty))} \lesssim 1 + E_c  .
\end{equation}
For convenience, denote
\[
  \| u \|_{\tilde{S}(I)} := \| \nabla u \|_{L^3_{x,t}(I \times \R^4)}.
\]
Under the assumption \eqref{A}, we can also obtain
\[
  \| v^{j} \|_{\tilde{S}(I^{j})} < \infty,
\]
and so similarly we have
\[
  \lim_{J \rightarrow \infty} \underset{n}{\overline{\text{lim} }} \hspace{0.2em} \| u_n^{J} \|_{\tilde{S}([0,+\infty))} < \infty . 
\]
To apply Proposition \ref{perturbation}, it suffices to show that $u_n^J$ asymptotically solves (\ref{CP}) in the sense that
\begin{equation*}   
\lim_{J \rightarrow \infty} \underset{n}{\overline{\text{lim} }} \hspace{0.2em} \|  \nabla [(\partial_t  - \Delta) u_n^J - F(u_n^J)] \|_{L_{t,x}^{\frac{3}{2}} ([0,+\infty) \times \mathbb{R}^4)} = 0 
\end{equation*}
which reduces (adding and subtracting the term  $F(\sum_{j=1}^J v_n^j)$ and using the triangle inequality) to proving
\begin{equation}   
  \lim_{J \rightarrow \infty} \underset{n}{\overline{\text{lim} }} \hspace{0.2em} \|  \nabla [\sum_{j=1}^J F(v_n^j) - F(\sum_{j=1}^J v_n^j)] \|_{L_{t,x}^{\frac{3}{2}} ([0,+\infty) \times \mathbb{R}^4)}  = 0 
\label{ena} \end{equation}
and 
\begin{equation}   
\underset{n}{\overline{\text{lim} }} \hspace{0.2em} \|  \nabla [ F(u_n^J - e^{t \Delta} w_n^J) - F(u_n^J)] \|_{L_{t,x}^{\frac{3}{2}} ([0,+\infty) \times \mathbb{R}^4)}  = 0.  
\label{dyo} \end{equation}
The following easy pointwise estimate will be of use: 
\begin{equation}  
  | \nabla [(\sum_{j=1}^J F(v_j) - F(\sum_{j=1}^J v_j) ] | \lesssim_J \sum_{i \neq j}  |\nabla v_j| |v_i|^2.
\label{pointwise} \end{equation}
We have shown that for all $j \geq 1$  and n large enough $v_n^j \in \tilde{S}([0,\infty)),$ so using property (\ref{orth})
\begin{equation*}  
  \underset{n}{\overline{\text{lim} }}  \| |v_n^j|^2 \nabla v_n^i \|_{L_{t,x}^\frac{3}{2}([0,\infty) \times \mathbb{R}^4)} = 0 
\end{equation*} for all $ i \neq j$; thus 
\begin{equation*}   
\underset{n}{\overline{\text{lim} }}  \| \nabla [(\sum_{j=1}^J F(v_j) - F(\sum_{j=1}^J v_j) ]  \|_{L_{t,x}^\frac{3}{2}} 
 \lesssim_J  \lim_{n \rightarrow \infty}  \underset{n}{\overline{\text{lim} }} \sum_{i \neq j} \| \nabla v_n^j |v_n^i|^{2} \|_{L_{t,x}^\frac{3}{2}} = 0 
\end{equation*}  
settling (\ref{ena}).

\begin{equation*} 
  \| \nabla [F(u_n^J - e^{t \Delta} w_n^J ) - F(u_n^J) ]  \|_{L_{t,x}^\frac{3}{2}}  \lesssim  \| \nabla e^{t \Delta} w_n^J \|_{L_{t,x}^{3} } \|  e^{t \Delta} w_n^J \|_{L_{t,x}^{6}}^{2}  
 +  \| |u_n^J|^{2} \nabla  e^{t \Delta} w_n^J \|_{L_{t,x}^{\frac{3}{2}}} + 
\end{equation*}
\begin{equation*}  
\| \nabla u_n^J  \|_{L_{t,x}^{3}}  \| e^{t \Delta} w_n^J \|_{L_{t,x}^{6}}^{2}  +   \| \nabla u_n^J  \|_{L_{t,x}^{3}}  \|  e^{t \Delta} w_n^J \|_{L_{t,x}^{6}} \| u_n^J \|_{L_{t,x}^{6}}. 
\end{equation*}
The first, third and fourth terms are easily seen to converge to zero (using the space-time estimates, the fact that $w_n^J$ is bounded in $\dot{H}^1$ and (\ref{holygrailH})), so (\ref{dyo}) is reduced to showing
\begin{equation*} 
\lim_{J \rightarrow \infty} \underset{n}{\overline{\text{lim} }} \|  |u_n^J|^{2} \nabla  e^{t \Delta} w_n^J   \|_{L_{t,x}^{\frac{3}{2}}} = 0.
\end{equation*}
By H\"{o}lder and the space-time estimates,
\begin{equation*}  \begin{split}
\||u_n^J|^{2} \nabla  e^{t \Delta} w_n^J \|_{L_{t,x}^{\frac{3}{2}}} &\lesssim \| u_n^J \|_{L_{t,x}^{6}}^{\frac{3}{2}}   \| \nabla  e^{t \Delta} w_n^J \|_{L_{t,x}^{3}}^{\frac{1}{2}}   \| u_n^J \nabla  e^{t \Delta} w_n^J   \|_{L_{t,x}^{2}}^{\frac{1}{2}}\\  &\lesssim \|  (\sum_{j=1}^J v_n^j ) \nabla  e^{t \Delta} w_n^J  \|_{L_{t,x}^{2}}^{\frac{1}{2}} +  \|  e^{t \Delta} w_n^J   \|_{L_{t,x}^{6}}^{\frac{1}{2}}    \|  \nabla  e^{t \Delta} w_n^J  \|_{L_{t,x}^{3}}^{\frac{1}{2}}   \\
&\lesssim  \|  (\sum_{j=1}^J v_n^j ) \nabla  e^{t \Delta} w_n^J \|_{L_{t,x}^{2}}^{\frac{1}{2}} +  \| e^{t \Delta} w_n^J  \|_{L_{t,x}^{6}}^{\frac{1}{2}}. \end{split}
\end{equation*} 
Again due to (\ref{holygrailH}) it suffices to prove
\begin{equation*}  
\lim_{J \rightarrow \infty} \underset{n}{\overline{\text{lim} }}   \|  (\sum_{j=1}^J v_n^j ) \nabla  e^{t \Delta} w_n^J \|_{L_{t,x}^{2}} = 0.  
\end{equation*} 
For any $\eta > 0 $ by summability, we see that there exists $J' = J' (\eta) \geq 1$ such that 
$\ds \sum_{j \geq J'} \| v_n^j \|_{S([0, \infty))} \leq \eta$.
For this $J'$,
\begin{equation*} 
\underset{n}{\overline{\text{lim} }}  \| \left (\sum_{j=J'}^J v_n^j \right ) \nabla  e^{t \Delta} w_n^J \|_{L_{t,x}^{2}}^{6}  \lesssim \underset{n}{\overline{\text{lim} }} \left ( \sum_{j \geq J'} \| v_n^j \|_{S([0, \infty))} \right)  \|  \nabla  e^{t \Delta} w_n^J  \|_{L_{t,x}^{3}}^{6} \lesssim \eta .
\end{equation*}
As $\eta > 0 $ is arbitrary, it suffices to show 
\begin{equation*}  
  \lim_{J \rightarrow \infty} \underset{n}{\overline{\text{lim} }}   \|  v_n^j  \nabla  e^{t \Delta} w_n^J \|_{L_{t,x}^{2}} = 0, \hspace{0.2em} 1 \leq j \leq J' .
\end{equation*}
Changing variables and assuming (by density) 
$v^j  \in C_c^{\infty} (\mathbb{R}^{+} \times \mathbb{R}^4)$, 
by \mbox{H\"{o}lder} and the scale-invariance of the norms, proving (\ref{dyo}) reduces to proving 
\begin{equation*}  
  \lim_{J \rightarrow \infty} \underset{n}{\overline{\text{lim} }}  \|  \nabla  e^{t \Delta} w_n^J \|_{L_{t,x}^{2} (K)} = 0,  
\end{equation*} 
for any compact $ K \in  \mathbb{R}^{+} \times \mathbb{R}^4.$ 
This result is the direct heat analogue of Lemma 2.5 in \cite{KV}.

We have verified all the requirements of the stability proposition (\ref{perturbation}), 
hence we conclude that 
 \begin{equation*}  
   \| u_n \|_{S([0,\infty))} \lesssim 1 + E_c \end{equation*} 
 contradicting (\ref{bad sequence}).

The problem now is that the kinetic energy is not conserved. The difficulty arises from the possibility that the S-norm of several profiles is large over short times, while their kinetic energy does not achieve the critical value until later. To finish the proof of proposition we have to prove that only one profile is responsible for the blow-up.

We can now (after possibly rearranging the indices) assume there exists $ 1 \leq J_1 < J_0$ such that
\begin{equation*}  \| v^j \|_{S(I^j)} = \infty ,  1 \leq j \leq J_1\hspace{0.3em} \text{and} \hspace{0.3em}  \| v^j \|_{S([0,\infty))} < \infty ,   j > J_1 \end{equation*}

Again, we follow the combinatorial argument of \cite{KV}:  for each integer $m,n \geq 1, $ define an integer $j=j(m,n) \in \{1,...,J_1\}$ and an interval $K_n^m$ of the form $[0,\tau]$ by 
\begin{equation} \sup_{1 \leq j \leq J_1} \|v_n^j \|_{S(K_n^m)} =  \|v_n^{j(m,n)} \|_{S(K_n^m)}= m. \label{K-definition} \end{equation}
By the pigeonhole principle, there is a $1 \leq j \leq J_1$ such that for infinitely many $m$ one has $j(m,n) = j_1$ for infinitely many n. Reordering the indices, if necessary, we may assume $j_1 =1.$
By the definition of the critical kinetic energy 
\begin{equation}  
\limsup_{m \rightarrow \infty} \limsup_{n \rightarrow \infty}  \sup_{t \in K_n^m} \| \nabla v^1_n(t)\|^2_2 \geq E_c. \label{one side} 
\end{equation}
    
By (\ref{K-definition}), all $v^j_n$ have finite S-norms on $K_n^m$ for each $m \geq 1.$ In the same way as before, we check again that the assumptions of 
Proposition~\ref{perturbation} are satisfied to conclude that for $J$ and $n$ large enough, $u^J_n$ is a good approximation to $u_n$ on $K_n^m.$ 
In particular we have for each $m \geq 1,$
\begin{equation}  \label{approximate2}
  \lim_{J \rightarrow \infty} \limsup_{n \rightarrow \infty} 
  \| u^J_n - u_n \|_{L^{\infty}_t \dot{H}^1_x (K_n^m  \times \mathbb{R}^4)} = 0. \end{equation}

\begin{lemma}   \label{later times} (Kinetic energy decoupling for later times). 
For all $J \geq 1$ and $m \geq 1,$ 
\begin{equation} \label{later_times} 
  \limsup_{n \rightarrow \infty} \sup_{t \in K_n^m} |  \|\nabla u^J_n(t) \|^2_2 
  - \sum^{J}_{j=1} \| \nabla v^j_n(t) \|^2_2 - \| \nabla w^J_n \|^2_2   |  = 0  
\end{equation}
\end{lemma}

\begin{proof} Fix $J \geq 1$ and $m \geq 1.$ Then, for all $t \in K_n^m,$
\begin{equation*} 
\| \nabla u^J_n (t) \|^2_2 = \; < \nabla u^J_n(t), \nabla u^J_n(t) > \; 
= \sum^{J}_{j=1} \| \nabla v^j_n(t) \|^2_2 + \| \nabla w^J_n \|^2_2 
\end{equation*}
\begin{equation*} 
+ \sum_{j \neq j'} < \nabla v^j_n(t), \nabla v^{j'}_n(t) > + 2 \sum_{j=1}^{J} <\nabla e^{t \Delta} w^J_n, \nabla v^j_n(t)>   .
\end{equation*} 
It suffices to prove (for all sequences $t_n \in K^m_n$) that
\begin{equation} 
  \label{vv} <\nabla v^j_n(t_n), \nabla v^{j'}_n(t_n)> \xrightarrow{n \to \infty} 0 
\end{equation}
and
\begin{equation} \label{wv} 
  <\nabla e^{t_n \Delta} w^J_n, \nabla v^j_n(t_n)> \xrightarrow{n \to \infty} 0 .
\end{equation}
Since $t_n \in K^m_n \subset [0, T^{n,j}_{\text{max}}),$ for all $1 \leq j \leq J_1, $ we have 
$t_{n,j} := \frac{t_n}{(\lambda^j_n)^2} \in I^j$ for all $j \geq 1.$ For $j > J_1$ the lifespan is $\mathbb{R}^{+}.$  
By refining the sequence using the standard diagonalization argument, we can assume that $t_{n,j}$ converges ($+\infty$ is also possible) for every $j.$\\

We deal with (\ref{vv}) first. If both $\ds t_{n,j}, t_{n,j'} \rightarrow \infty,$ necessarily $\ds j,j' > J_1$ and $v^j, v^{j'}$ are global solutions satisfying the kinetic energy bound (\ref{condition}), so by  Theorem (\ref{Decay Theorem}) $\ds \| v^j \|_{\dot{H}^1}, \|v^{j'}\|_{\dot{H}^1} \xrightarrow {t \to \infty} 0.$ Employing H\"older's inequality and the scaling invariance of the $\dot{H}^1$-norm, we get (\ref{vv}) for this case. When $\ds t_{n,j} \rightarrow \infty$ but $\ds t_{n,j'} \rightarrow \tau_{j'}:$ using the continuity of the flow in $\ds \dot{H}^1$ we can, for the limit, replace $\ds \nabla \{ \frac{1}{\lambda^{j'}_n} v^{j'}(t_{n,j'},\frac{x-x_n^j}{\lambda^{j'}_n}) \}$ with $\ds \nabla \{ \frac{1}{\lambda^{j'}_n} v^{j'}(\tau_{j'},\frac{x-x_n^j}{\lambda^{j'}_n}) \}.$ By an $L^2$- approximation, we can also assume we are working with smooth, compactly supported functions. In this case, we can bound $\ds <\nabla v^j_n(t_n), \nabla v^{j'}_n(t_n)>$ by  $\ds \| v^j(t_{n,j}) \|_{\dot{H}^1} \| v^{j'}(\tau_{j'}) \|_{\dot{H}^1} \rightarrow 0,$ as $n \rightarrow \infty.$ The remaining case is when both $ t_{n,j}$ and $t_{n,j'}$ converge to finite $\ds \tau_j, \tau_{j'}$ in the interior of $I^j,I^{j'}$ respectively. We can replace as above $\ds t_{n,j},t_{n,j'}$ by $\tau_j, \tau_{j'}$ respectively, and perform a change of variables: 
\[
  \ds <\nabla v^j_n(t_n), \nabla v^{j'}_n(t_n)> = \int (\frac{\lambda^{j}_n}{\lambda^{j'}_n} )^2  \nabla v^j(\tau_j,x), \nabla v^{j'}(\tau_{j'}, \frac{\lambda^{j}_n}{\lambda^{j'}_n}x + \frac{x_n^j-x_n^{j'}}{\lambda^{j'}_n}) dx
\] 
which is going to zero assuming, without loss of generality that $\ds \frac{\lambda^{j}_n}{\lambda^{j'}_n} \rightarrow 0$ and the functions in the integrand 
are compactly supported, thus concluding the case (\ref{vv}).

For the case (\ref{wv}), perform a change of variable: 
\[
  \ds <\nabla e^{t_n \Delta} w^J_n, \nabla v^j_n(t_n)>
  =<\nabla e^{t_{n,j} \Delta} [\lambda^{j}_n w^J_n(\lambda^{j}_n x+x_n^{j})] , \nabla v^j(t_{n,j}) >.
\]
When $t_{n,j} \rightarrow \infty,$ using H\"older, the heat estimates (\ref{STdecayheat}) (and the boundedness of $w^J_n$ in $\dot{H}^1$ coming from the profile decomposition) and 
Theorem~\ref{Decay Theorem} as before, we get to the result.
For the case $\ds t_{n,j} \rightarrow \tau_j < + \infty,$ we can, as before, replace $t_{n,j}$ by its limit $\tau_j$ in the integral $\ds \int \nabla e^{t_{n,j} \Delta} [\lambda^{j}_n w^J_n(\lambda^{j}_n x+x_n^{j})] \cdot \nabla v^j(\tau_{j}, x) dx.$ 
Using (\ref{holygrailH}) and (\ref{STdecayheat}) we can see that $\ds e^{t_{n,j} \Delta} [\lambda^{j}_n w^J_n(\lambda^{j}_n x+x_n^{j})] \rightharpoonup 0$ in $\dot{H}^1$, 
which concludes the proof of the case (\ref{wv}) and hence the proof of the Lemma. 
\end{proof}

By (\ref{bad sequence}), (\ref{approximate2}), (\ref{later_times}), we get
\begin{equation*} \begin{split}
  E_c &\geq  \limsup_{n \rightarrow \infty} \sup_{t \in K_n^m} \|\nabla u_n(t) \|^2_{L^2} \\ &= \lim_{J \rightarrow \infty} \limsup_{n \rightarrow \infty} \{ \|\nabla w^J_n(t) \|^2_{L^2} +  \sup_{t \in K_n^m} \sum^J_{j=1} \|\nabla v^j_n(t) \|^2_{L^2} \}. \end{split}
\end{equation*}

Taking a limit in $m$ and employing (\ref{one side}), we see that we actually have equality everywhere. This implies that $J_1 =1, v^j_n \equiv 0, \forall j \geq 2, w_n := w^1_n \xrightarrow{ \dot{H}^1} 0.$ So $u_n(0,x) = \frac{1}{\lambda_n} \phi(\frac{x-x^1_n}{\lambda^1_n}) + w_n(x), $ for some functions $\phi, w_n \in \dot{H}^1, w_n \xrightarrow{s} 0 \hspace{0.2em} \text{in}\hspace{0.2em} \dot{H}^1.$

Thus we have shown that for the sequence of initial data 
$u_{n,0}$ that 
\[
  \lambda^1_n u_{n,0} (\lambda^1_n x+x_n^{1}) \xrightarrow{\dot{H}^1} \phi.
\] 
This finishes the proof of Proposition \ref{Palais-Smale}.
\end{proof}

Now, we are in a position to prove Theorem~\ref{critical thm}.
\begin{proof} 
By the definition of $E_c$ we can find a sequence of solutions $u_n : I_n \times \mathbb{R}^4 \rightarrow \mathbb{R},$ with $I_n$ compact, so that \begin{equation*} \sup_n \sup_{t \in I_n} \|\nabla u_n(t) \|^2_{L^2} = E_c
\hspace{0.3em} \text{and} \hspace{0.3em} \lim_n \|u_n\|_{S(I_n)} = + \infty . \end{equation*} 
An application of Proposition~\ref{Palais-Smale} shows that the corresponding 
sequence of initial data converges strongly, modulo symmetries, to some 
$\phi \in  \dot{H}^1$.
By rescaling and translating $u_n$, we may in fact assume 
$u_{n,0} := u_n(0,\cdot) \xrightarrow{\dot{H}^1} \phi$.

Let $u_c : I \times \mathbb{R}^4 \rightarrow \mathbb{R}$ be the maximal-lifespan solution with initial data $\phi.$
Since $u_{n,0} \xrightarrow{\dot{H}^1} \phi,$ employing the stability Proposition \ref{perturbation}, $I \subset \liminf I_n,$
and $\|u_n - u_c \|_{L^{\infty}_t \dot{H}^1_x (K \times \mathbb{R}^4)} \xrightarrow{n \rightarrow \infty} 0, $ for all compact $K \subset I.$
Thus, by (\ref{bad sequence}): \begin{equation} \sup_{t \in I} \| \nabla u_c(t) \|^2_{L^2} \leq E_c. \end{equation}
Applying the stability Proposition~\ref{perturbation} once again we can also see that 
$\| u_c \|_{S(I)} = \infty$. 
Hence, by the definition of the critical kinetic energy level, $E_c,$ 
\begin{equation} 
  \sup_{t \in I} \| \nabla u_c(t) \|^2_{L^2} \geq E_c .
\end{equation}
In conclusion, 
\begin{equation} 
  \sup_{t \in I} \| \nabla u_c(t) \|^2_{L^2} = E_c  \label{=Ec}
\end{equation}
and 
\begin{equation}
  \|u_c\|_{S(I)} = +\infty. \label{U-BU}
\end{equation}
Finally, the compactness modulo symmetries~\eqref{precompact}
follows from another application of Proposition \ref{Palais-Smale}. 
We omit the standard proof (see for example \cite{KMa} or \cite{KV}).
\end{proof}

\section{Rigidity} \label{Rigidity Section}

The main result of this section is the following theorem ruling out finite-time blowup of compact (modulo symmetries) solutions. Note this is a considerably stronger
statement than we require,
since it is not limited to solutions with below-threshold kinetic energy:
\begin{theorem} \label{Rigidity Theorem} 
If $u$ is a solution to~\eqref{CP} on maximal existence interval $I = [0,T^*)$, 
such that $\ds K:= \left\{ \frac{1}{\lambda(t)} u(t, \frac{x-x(t)}{\lambda(t)}) \; | \;  t \in I \right\}$
is precompact in $\dot{H}^1$ for some $\ds x(t) \in \mathbb{R}^4$, $\lambda(t) \in \mathbb{R}^+$, 
then $T^* = + \infty$. 
\end{theorem}

As a corollary, we can complete the proof of the main result Theorem~\ref{Main Theorem} by showing:
\begin{corollary} \label{Global Theorem}
For any solution satisfying~\eqref{result finite}, $T_{max}(u(0)) = \infty$.
\end{corollary}
\begin{proof}
By Theorem~\ref{Rigidity Theorem}, the solution $u_c$ produced by Theorem~\ref{critical thm} must be global:
$\ds T_{max}(u_c(0)) = \infty$. But since $\ds \| u_c \|_{S(\R^+)} = \infty$, 
Theorem~\ref{Decay Theorem} shows $E_c = \| \nabla W \|_2^2$, and the Corollary follows.
\end{proof}

The rest of the section is devoted to the proof of the Theorem~\ref{Rigidity Theorem}.
Our proof is inspired by the work of Kenig and Koch \cite{KK} for the Navier-Stokes 
system, and it's based on classical parabolic tools -- local smallness regularity, backwards uniqueness, and unique continuation -- though implemented in a somewhat different way. 
In particular, we will make use of the following two results, proved in 
\cite{ESS1}, \cite{ESS2} (also see \cite{ESS3}):

\begin{theorem} (Backwards Uniqueness) \label{Backwards} 
Fix any $R, \delta, M,$ and $ c_0 > 0.$ 
Let $Q_{R,\delta} := (\mathbb{R}^4\setminus B_R(0)) \times (-\delta,0),$ 
and suppose a vector-valued function $v$ and its distributional derivatives satisfy 
$v, \nabla v, \nabla^2 v \in L^2(\Omega)$ for any bounded subset 
$\Omega \subset Q_{R,\delta}$, 
$|v(x,t)| \leq e^{M|x|^2}$ for all $(x,t) \in Q_{R,\delta}$, 
$|v_t- \Delta v| \leq c_0 ( |\nabla v| + |v|)$ on $Q_{R,\delta}$, 
and $v(x,0)=0$ for all $x \in \mathbb{R}^4\setminus B_R(0)$. 
Then $ v \equiv 0 $ in $Q_{R,\delta}.$
\end{theorem}

\begin{theorem} (Unique Continuation)\label{unique_c} 
Let $Q_{r,\delta} :=  B_r(0) \times (-\delta,0),$ for some $r, \delta > 0,$  and suppose a vector-valued function $v$ and its distributional derivatives satisfy 
$v, \nabla v, \nabla^2 v \in L^2(Q_{r,\delta})$ and there exist $c_0, C_k > 0, (k \in \mathbb{N})$ such that $|v_t - \Delta v | \leq c_0  ( |\nabla v| + |v|)$ a.e. on $Q_{r,\delta}$ and 
$ |v(x,t)| \leq C_k (|x| + \sqrt{-t})^k$ for all $(x,t) \in Q_{r,\delta}.$ 
Then $v(x,0) \equiv 0$ for all $x \in B_r(0).$
\end{theorem}

 As well, we establish the following:
\begin{lemma} (Local Smallness Regularity Criterion) \label{regularity}
For any $k \in \mathbb{N}$, there are $\epsilon_0 > 0$ and $C$ such that: 
if $u$ is a solution of equation (\ref{CP}) on $Q_1$, 
where $Q_r := B_r(0) \times (-r^2,0)$ for $r>0,$ and  satisfies  
\begin{equation*}  
  \epsilon := \|u\|_{L_t^{\infty} (\dot{H}_x^1 \cap L_x^4) (Q_1)} < \epsilon_0 
\end{equation*}
then $u$ is smooth on $\overline{Q_{\frac{1}{2}}}$ with bounds
\begin{equation*}
  \max_{\overline{Q_{\frac{1}{2}}}} | D^k u | \leq C \epsilon. 
\end{equation*}
\end{lemma}
\begin{proof}
Assume $\|u\|_{L_t^{\infty} (\dot{H}_x^1 \cap L_x^4) (Q_1)} < \epsilon, $ for 
$\epsilon$ small enough (to be picked).
Define  
 \[
   \| u \|^2_{X(Q_1)} := \| \nabla u \|^2_{L^{\infty}_t L^2_x \cap L^2_t L^4_x (Q_1)} 
   + \| u \|^2_{L^{\infty}_t L^4 (Q_1) } + \| D^2 u\|^2_{L^2_t L^2_x(Q_1)}.
\]
Assuming for ease of writing that $u$ is real-valued, differentiating~\eqref{CP} 
and defining $\tilde{u} := \nabla u$, we get
\begin{equation} \label{d1eqn}
  \tilde{u}_t = \Delta \tilde{u} + 3 u^2 \tilde{u}.
\end{equation}
Consider a smooth, compactly supported spatial cut-off function $\phi_0(x)$ such that 
$\text{supp}(\phi_0) \subset B_1(0)$ and $\phi_0 \equiv 1$ on $B_{\rho_0}(0)$,
for some $\frac{1}{2} < \rho_0 < 1$ to be chosen. 
Multiplying the above equation by $\phi^2_0 \tilde{u}$ and integrating in space-time 
(from now on, unless otherwise specified, $t \in [-1,0]$): 
\[
\begin{split}
  &\ds \int_{-1}^t \int_{|x| \leq 1} (\tilde{u}_t - \Delta \tilde{u}) \phi^2_0 \tilde{u}\hspace{0.2em}  dx dt = 3 \int_{-1}^t \int_{|x| \leq 1} (u^2 \tilde{u}) \phi^2_0 \tilde{u} \hspace{0.2em} dx dt \\
  & \qquad \Rightarrow \frac{1}{2} \| \phi_0 \tilde{u}(t) \|^2_{L^2} + \int_{-1}^t \int_{|x| \leq 1} \phi^2_0 |\nabla \tilde{u}|^2 \hspace{0.2em} dx dt \\
  & \quad = \frac{1}{2} \| \phi_0 \tilde{u}(0) \|^2_{L^2} + 3 \int_{-1}^t \int_{|x| \leq 1} \phi^2_0 u^2 \tilde{u}^2 dx dt 
  \ds + 2 \int_{-1}^t \int_{|x| \leq 1} \phi_0 \nabla \phi_0 (\tilde{u} \nabla \tilde{u}) dx dt .
\end{split}
\]
For the sake of brevity, let us define $v_0 := \phi_0 \tilde{u} = \phi_0 \nabla u$ 
and thus (always on the same cylinder):
\begin{equation*}
\begin{split}
  \|v_0\|^2_{L^\infty_t L^2_x} + \| \phi_0 \nabla \tilde{u} \|^2_{L^2_t L^2_x} &\lesssim 
  \|v_0 (0,x) \|^2_{L^2} + \| u^2 \|_{L^{\infty}_t L^2_x } \| v^2_0 \|_{L^1_t L^2_x } + \| \phi_0 \nabla \tilde{u} \|_{L^2_t L^2_x} \|\tilde{u} \|_{L^2_t L^2_x} \\
  &=\|v_0 (0,x) \|^2_{L^2} + \| u \|^2_{L^{\infty}_t L^4_x } \| v_0 \|^2_{L^2_t L^4_x } + \| \phi_0 \nabla \tilde{u} \|_{L^2_t L^2_x } \|\tilde{u} \|_{L^2_t L^2_x }.
\end{split}  
\end{equation*}
By the smallness assumed on the cylinder $Q_1$ and an application of Young's inequality, 
for any $\delta > 0$ (and also using H\"older and the boundedness of the domain):  
\begin{equation*}
\|v_0\|^2_{L^\infty_t L^2_x} + \| \phi_0 \nabla \tilde{u} \|^2_{L^2_t L^2_x} \lesssim \epsilon^2 +\epsilon^2 \| v_0 \|^2_{L^2_t L^4_x (Q_1)} +\delta^2 \| \phi_0 \nabla \tilde{u} \|_{L^2_t L^2_x (Q_1)}^2 +  \frac{\|\tilde{u} \|^2_{L^\infty_t L^4_x (Q_1)}}{\delta^2} 
\end{equation*}
\begin{equation*} 
  \Rightarrow \|v_0\|^2_{L^\infty_t L^2_x} + \| \phi_0 \nabla \tilde{u} \|^2_{L^2_t L^2_x} \lesssim \epsilon^2 + \frac{\epsilon^2}{\delta^2} + \epsilon^2 \| v_0 \|^2_{L^2_t L^4_x}  
\end{equation*}
if $\delta$ is chosen small enough.
Since $\nabla v_0 = \phi_0 \nabla \tilde{u} + \nabla \phi_0 \hspace{0.2em} \tilde{u}:$
\begin{equation*} 
  \| \nabla v_0 \|_{L^2} \lesssim \| \phi_0 \nabla \tilde{u} \|_{L^2} + 
  \| \nabla \phi_0 \|_{L^4} \| \tilde{u} \|_{L^4} 
\end{equation*} 
and so using the Sobolev inequality,
\begin{equation*} 
  \| v_0 \|^2_{L^{\infty}_t L^{2}_x} + \| \nabla v_0 \|^2_{L^2_t L^2_x} + \| v_0 \|^2_{L^2_t L^4_x} 
  \lesssim \epsilon^2 + \epsilon^2 \| v_0 \|^2_{L^2_t L^4_x}  .
\end{equation*}
Choosing $\e$ small enough yields 
\[
  \| u \|_{X(Q_{\rho_0})} \lesssim \epsilon . 
\]

Define another smooth compactly supported cut-off function $\phi_1(x) \leq \phi_0(x)$, 
with support in $B_{\rho_0}$, and $\phi_1 \equiv  1$ on $B_{\rho_1}(0)$, some 
$\frac{1}{2} < \rho_1 < \rho_0 < 1$ to be chosen. 
Let $\hat{v} := D^2 u$, and $v_1 := \phi_1 \hat{v}$. 
\begin{remark}
We will be abusing notation from this point onwards. For the pointwise operations and estimates we are actually considering the mixed partial derivatives $\partial_k \partial_j u, j,k=1,...,4$ but we will be writing $D^2 u$ all the same without taking care to specify the matrix element at hand. In the end, we are using standard matrix norms.
\end{remark}
Differentiating~\eqref{d1eqn}, multiplying by $\phi_1^2 \hat{v}$, and integrating over space gives 
\begin{equation} \label{spaceint}
\begin{split}
  \frac{1}{2} \partial_t \int \phi_1^2 \hat{v}^2 dx +
  \int \phi_1^2 |\nabla \hat{v}|^2 dx &= 3 \int \phi_1^2 u^2 \hat{v}^2 dx \\ 
  &\quad + 
  6 \int \phi_1^2 u \tilde{u}^2 \hat{v} dx +
  2 \int \phi_1 \nabla \phi_1 \cdot \hat{v} \nabla \hat{v} dx.
\end{split}
\end{equation}
Since by the previous step, $ \| \nabla v_0 \|_{L^2 L^2(Q_{\rho_0})} \lesssim \e$,
we can find $-1 < t_1 < -\rho_0^2$ such that
$\| \nabla v_0 (\cdot, t_1)\|_{L^2(B_{\rho_0})} \lesssim \epsilon$
(where the implied constant may depend on $\rho_0$), so that
\[
  \| \phi_1 \hat{v}(\cdot, t_1) \|_{L^2} = \| \phi_1 D^2 u(\cdot, t_1) \|_{L^2} \leq
  \| \nabla v_0(\cdot, t_1)\|_{L^2(B_{\rho_0})} \lesssim \epsilon.
\]
Integrating~\eqref{spaceint} in $t$ from $t_1$ to $0$, and using the estimates
from the previous step:
\[
\begin{split}
  \| v_1 \|_{L^\infty_t L^2}^2 + \| \phi_1 \nabla \hat{v} \|_{L^2_t L^2_x}^2 
  &\lesssim \| \phi_0 u\|_{L^{\infty} L^4}^2 \|v_1\|_{L^2 L^4}^2 
  + \| \hat{v} \nabla \hat{v} \phi_1 \nabla \phi_1 \|_{L^1_t L^1_x} + \e^2 \\
 &\quad + \| \phi_0 u \|_{L^{\infty} L^4}  \| v_0 \|_{L^{\infty} L^4}  \|v_0 \|_{L^2 L^4} \|v_1\|_{L^2 L^4}  \\ 
  & \lesssim \e^2 \| v_1 \|_{L^2 L^4}^2 + \e^3 \|v_1\|_{L^2 L^4}\\ &\quad+ 
  \| \phi_1 \nabla \hat{v} \|_{L^2 L^2} \|  \nabla \phi_1 \hat{v} \|_{L^2 L^2} + \e^2
\end{split}
\]
where everywhere here the time interval is $[t_1,0]$.
We have 
\[
  \nabla v_0 = \phi_0  D^2 u + \nabla \phi_0 \nabla u = \phi_0 \hat{v} + \nabla \phi_0 \tilde{u}
\] 
and so
\[
  | \phi_0 \hat{v} | \lesssim | \nabla v_0 | + | \nabla \phi_0 \tilde{u}|  \;
  \Rightarrow \; | \nabla \phi_1 \hat{v} | \lesssim |\frac{\nabla \phi_1}{\phi_0}| | \phi_0 \hat{v} | 
  \lesssim | \nabla v_0 | + |\nabla \phi_0 \tilde{u}|.
\]
Thus
\[
  \| \nabla \phi_1 \hat{v} \|_{L^2 L^2} \lesssim \| \nabla v_0 \|_{L^2 L^2} + \epsilon \lesssim \epsilon.
\] 
By Young's  inequality once more, for some $\delta_1 > 0$ sufficiently small,
\begin{equation*}
  \| v_1 \|^2_{L^{\infty} L^2} + \| \phi_1 \nabla \hat{v} \|^2_{L^2 L^2}  
 \lesssim \epsilon^2 \| v_1 \|^2_{L^2 L^4}  + 
 \delta_1^2 \| \phi_1 \nabla \hat{v} \|^2_{L^2 L^2} + 
  \delta_1^2 \|v_1 \|^2_{L^2 L^4} + \frac{\epsilon^2}{\delta_1^2}.
\end{equation*}
Using Sobolev again as above, 
$\| v_1 \|^2_{L^{\infty} L^2} + \|v_1\|^2_{L^2 L^4} + \| \nabla v_1 \|^2_{L^2 L^2} \lesssim \epsilon.$ 
In particular 
\[
  \|D^2 u \|_{X(Q_{\rho_1})} \lesssim \epsilon.
\]

This process can be iterated a given finite number of times, to show that for given
$k > 0$, there are $\epsilon_0 = \epsilon_0(k)$, $C = C(k)$, such that if 
$\| u \|_{L^{\infty} (\dot{H}^1 \cap L^4) (Q_1)} = \epsilon < \epsilon_0$,
then $\| D^k u \|_{X(Q_{1/2})} \leq C \epsilon.$
\end{proof}

  We proceed now with the proof of Theorem~\ref{Rigidity Theorem}.
\begin{proof}
Let us assume that the conclusion is false, i.e., $\ds T^{*}  < +\infty.$
Note first that 
\[
   \lambda(t) \rightarrow +\infty.
\] 
In fact,  $\ds \ds \liminf_{t \to T^*-} \sqrt{T^*-t} \; \lambda(t) > 0$,
since if $\sqrt{T^*-t_n} \lambda(t_n) \to 0$ along
a sequence $t_n \nearrow T^*$,
by the compactness assumption (and up to subsequence) 
\[
  v_n(x) := \frac{1}{\lambda(t_n)} u_c(t_n,\frac{x-x(t_n)}{\lambda(t_n)})
  \xrightarrow{\dot{H}^1} \; \exists \; v(x) \in \dot H^1.
\]
Let $\hat T > 0$ be the maximal existence time for 
the solution of the Cauchy problem~\eqref{CP} with initial data $v(x)$.
Define $w_n(t,x)$ to be the solutions with initial data 
$w_n(x,t_n) = v_n(x)$ prescribed at time $t_n$, and denote their maximal lifespans as 
$[t_n, T^{\text{max}}_n)$. By continuous dependence on initial data, 
$\ds 0 < \hat{T} \leq \liminf ( T^{\text{max}}_n - t_n )$. 
But from scaling: 
\[
\begin{split}
  T^{\text{max}}_n - t_n &= T^{\text{max}} \left( \frac{1}{\lambda(t_n)} u_c (t_n, \frac{\cdot-x(t_n)} 
  {\lambda(t_n)}) \right) = \lambda^2(t_n) T^{\text{max}}(u_c(t_n,\cdot) ) \\
  &= \lambda^2(t_n) (T^*-t_n) \to 0,
\end{split}
\]
a contradiction.

By compactness in $\dot{H}^1$, and the continuous embedding 
$\dot{H}^1 \hookrightarrow L^4,$ for every $\epsilon > 0,$ there is a $R_{\epsilon} > 0$ 
such that for all $t \in I:=[0,T^*) :$
\begin{equation} \label{out} 
  \int_{|x-x(t)| \geq \frac{R_{\epsilon}}{\lambda(t)}} \left(
  |\nabla u_c (t,x)|^2  + |u_c (t,x)|^4 \right) dx < \epsilon.
\end{equation}

Fix any $\{t_n\} \subset [0, T^*), t_n \nearrow T^*,$ and let 
$\lambda_n = \lambda(t_n) \rightarrow \infty$ and 
$\{x_n\} = \{ x(t_n) \} \subset \mathbb{R}^4$, so that (up to subsequence)
\begin{equation*} 
  v_n(x) =  \frac{1}{\lambda_n} u_c(\frac{x-x_n}{\lambda_n}, t_n) 
  \xrightarrow{\dot{H^1}} \bar{v}, \mbox{ some } \bar{v} \in \dot{H}^1,
\end{equation*}
and also in $L^4$ by Sobolev embedding.

We also make and prove the following claim as in \cite{KK}
\begin{claim}
For any $R>0,$ \label{R-claim} 
\begin{equation*}
\lim_{n \rightarrow \infty} \int_{|x| \leq R} |u_c (x, t_n)|^2 dx = 0.
\end{equation*}
\end{claim}
\begin{proof} 
\begin{equation*} \begin{split} \int_{|x| \leq R} |u_c(x, t_n)|^2 dx &=  
\int_{|x| \leq R} |\lambda_n v_n(\lambda_n x + x_n)|^2 dx \\ &= 
\frac{1}{\lambda^2_n} \int_{|y-x_n| \leq \lambda_n R} |v_n(y)|^2 dy = 
\frac{1}{\lambda^2_n} \| v_n \|^2_{L^2(B_{\lambda_n R}(x_n))}.
\end{split} 
\end{equation*} 
Denoting $B_r := B_r(0)$, for any $\epsilon > 0$, 
\begin{equation*} 
\frac{1}{\lambda^2_n} \| v_n \|^2_{L^2(B_{\lambda_n R}(x_n))} 
=  \frac{1}{\lambda^2_n} \| v_n \|^2_{L^2(B_{\lambda_n R}(x_n))
\cap B_{\epsilon \lambda_n R})} +  
\frac{1}{\lambda^2_n} \| v_n \|^2_{L^2(B_{\lambda_n R}(x_n)) \cap B^{c}_{\epsilon \lambda_n R})}. 
\end{equation*} 
Using H\"older's inequality, and the compactness, we get 
\begin{equation*}
\begin{split}
  \frac{1}{\lambda^2_n} \| v_n \|^2_{L^2(B_{\lambda_n R}(x_n))} & \lesssim 
 \frac{1}{\lambda^2_n} \| v_n \|^2_{L^4(B_{\lambda_n R}(x_n) \cap B^{c}_{\epsilon \lambda_n R})} | B_{\lambda_n R}(x_n) |^{\frac{1}{2}}\\ &\quad+ \frac{1}{\lambda^2_n} \| v_n \|^2_{L^4(\mathbb{R}^4)} |B_{\epsilon \lambda_n R}|^{\frac{1}{2}} \\
   &\lesssim \frac{| B_{\lambda_n R}(x_n) |^{\frac{1}{2}}}{\lambda^2_n}\left( \| v_n - \bar{v} \|^2_{L^4(\R^4)} + \| \bar{v} \|^2_{L^4(B_{\lambda_n R}(x_n) \cap B^{c}_{\epsilon \lambda_n R})}\right)\\
   &\quad+ \epsilon^2 R^2 \|\bar{v}\|^2_{L^4(\mathbb{R}^4)}\\
   &\lesssim R^2\| v_n - \bar{v} \|^2_{L^4(\R^4)}+R^2  \|\bar{v}\|^2_{L^4( B^{c}_{\epsilon \lambda_n R})} + \epsilon^2 R^2 \|\bar{v}\|^2_{L^4(\mathbb{R}^4)}  \\
   & \lesssim R^2 \| v_n - \bar{v} \|^2_{L^4(\R^4)}+ R^2  \|\bar{v}\|^2_{L^4( B^{c}_{\epsilon \lambda_n R})}+ \epsilon^2 R^2 .
\end{split}   
\end{equation*}
The first term goes to zero (as $n \rightarrow \infty$) because of the compactness, the second goes to zero (for fixed $\epsilon$) since $\lambda_n \rightarrow \infty,$ and the last one is arbitrarily small with $\epsilon.$ 
\end{proof}

We also prove that the center of compactness $x(t)$ is bounded:
\begin{proposition}
$\ds \sup_{0 \leq t < T^*} |x(t)| < \infty$.
\end{proposition}
\begin{proof}
We will first make the assumption that
\begin{equation} \label{EbarAss}
  \underline{E} := \inf_{t \in [0,T^*) } E(u_c(t))> 0,
\end{equation}
and later show that this is indeed the case for compact blowing-up solutions, without any size restriction. Note that under the assumptions of our Theorem \ref{Main Theorem}, i.e., in the below threshold case, we certainly have that $ \underline{E} > 0.$ This can be easily deduced by the variational estimates in Lemma \ref{VarLemma} and the small data theory.

The energy dissipation relation 
\begin{equation} 
  E(u(t_2))  + \int_{t_1}^{t_2} \|u_t\|^2_{L^2} \; ds  = E(u(t_1)) \leq E(u(0))
\label{E-dissip} 
\end{equation} 
for $t_2 > t_1 > 0$ will be of use.
We will assume for contradiction that there is a sequence of times 
$\ds t_n \nearrow T^*: |x(t_n)| \rightarrow \infty.$

Choose a smooth cut-off function $\psi$ such that
\begin{equation*}
\psi(r) = \left\{
\begin{array}{rl}
0 & \text{if } r \leq 1\\
1 & \text{if } r \geq 2
\end{array} \right.
\end{equation*}
and define $\ds \psi_R(x) := \psi(\frac{|x|}{R}).$
Choosing any $t_0 \in (0,T^*),$ we can find $R_0 \geq 1$ such that
\begin{equation}
\int_{\mathbb{R}^4} \left( \frac{1}{2} |\nabla u_c(t_0)|^2 - \frac{1}{4} (u_c(t_0))^4  \right) \psi_{R_0}(x) \hspace{0.1em} dx  \leq \frac{1}{4}  \underline{E} . \label{1/4} 
\end{equation}
Since $|x(t_n)| \rightarrow \infty$ and $\lambda(t_n) \to \infty$, for any $\epsilon > 0$,
$B_{\frac{R \epsilon}{\lambda(t_n)} } (x(t_n)) \subset B^{c}_{2R_0}$ 
for $n$ large enough, and so by~\eqref{out}:
\begin{equation*} 
  \lim_{t \nearrow T^*} \int_{\mathbb{R}^4} \left( \frac{1}{2} |\nabla u_c(t)|^2 
  - \frac{1}{4} (u_c(t))^4  \right) \psi_{R_0}(x) dx = \underline{E}, 
\end{equation*}
hence we can find a $t_1 \in (t_0,T^*)$ such that
\begin{equation}
\int_{\mathbb{R}^4} \left( \frac{1}{2} |\nabla u_c(t_1)|^2 - \frac{1}{4} (u_c(t_1))^4  \right) \psi_{R_0}(x) \hspace{0.1em} dx \geq \frac{1}{2}  \underline{E}.  
\label{1/2} 
\end{equation}
Combining (\ref{1/4}) and (\ref{1/2}):
\begin{equation} \int^{t_1}_{t_0} \frac{d}{dt} \int_{\mathbb{R}^4} \left( \frac{1}{2} |\nabla u_c(t)|^2 - \frac{1}{4} (u_c(t))^4  \right) \psi_{R_0}(x) \hspace{0.1em} dx dt \geq \frac{1}{4}  \underline{E} . \label{dt1}  
\end{equation} 
On the other hand: 
\begin{equation*}
\begin{split}
\frac{d}{dt} \int_{\mathbb{R}^4} & \left( \frac{1}{2} |\nabla u_c(t)|^2 - \frac{1}{4} (u_c(t))^4  \right) \psi_{R_0}(x) \hspace{0.1em} dx 
= \int_{\mathbb{R}^4} \left( \nabla u_c \cdot \nabla (u_c)_t - u_c^3 (u_c)_t \right) \psi_{R_0}(x) \hspace{0.1em} dx \\
&= \int_{\mathbb{R}^4} \left( \nabla u_c \cdot \nabla (u_c)_t - ((u_c)_t - \Delta u_c) (u_c)_t \right) \psi_{R_0}(x) \hspace{0.1em} dx\\
&=-\int_{\mathbb{R}^4} (u_c)^2_t \psi_{R_0} \hspace{0.1em} dx - \int_{\mathbb{R}^4} (u_c)_t \nabla u_c \cdot \nabla \psi_{R_0} \hspace{0.1em} dx
\lesssim \int_{\R^4} |(u_c)_t| |\nabla u_c| \; dx,
\end{split}
\end{equation*}
since $\ds | \nabla \psi_{R_0} (x) | \lesssim \frac{1}{R_0} \leq 1$. So by H\"older,
\begin{equation}
\begin{split} 
\int^{t_1}_{t_0} \frac{d}{dt} \int_{\mathbb{R}^4} & \left( \frac{1}{2} |\nabla u_c(t)|^2 - \frac{1}{4} (u_c(t))^4  \right) \psi_{R_0}(x) \hspace{0.1em} dx dt \\
&\lesssim 
\| \nabla u_c \|_{L^\infty_t L^2} \; \sqrt{t_1 - t_0} \; \|(u_c)_t\|_{L^2 L^2 [t_0,t_1] \times \mathbb{R}^4} \\ &\lesssim  \|(u_c)_t\|_{L^2 L^2( [t_0,T^*) \times \mathbb{R}^4)}           
\end{split}
\label{dt2}
\end{equation}
where we have uniformly bounded the kinetic energy of $u_c$ by once more employing the compactness. Combining (\ref{dt1}) and (\ref{dt2}) yields:
\begin{equation} 
0 < \hspace{0.3em} \frac{1}{4} \underline{E} \hspace{0.3em} \lesssim \|(u_c)_t\|_{L^2 L^2( [t_0,T^*) \times \mathbb{R}^4)} \to 0 \mbox{ as } t_0 \nearrow T^*
\label{x-contradiction} 
\end{equation}
by the energy dissipation relation (\ref{E-dissip}), a contradiction.

Now we show~\eqref{EbarAss}.
Choose a smooth cut-off function $\phi$ such that
\begin{equation*}
\phi(r) = \left\{
\begin{array}{rl}
1 & \text{if } r \leq 1\\
0 & \text{if } r \geq 2
\end{array} \right.
\end{equation*}
and define $\phi_R (x):= \phi(\frac{|x|}{R})$, and
\[
  I_R(t) := \frac{1}{2} \int (u_c(x,t))^2 \phi_R (x) dx, \qquad t \in [0, T^*).
\]
We then have 
\begin{equation*}\ds I'_R(t) =  \int \phi_R ((u_c)^4 - |\nabla u_c|^2) dx 
- \frac{1}{R} \int u_c \nabla u_c \cdot \nabla \phi(\frac{x}{R}) dx 
\end{equation*} 
and by Sobolev, Hardy and the compactness, we can immediately deduce that
\begin{equation*} 
  |I'_R(t)| \leq C,
\end{equation*}
$C$ a constant. Integrating from $t_0$ to $T^{*} > t > t_0 \geq 0$:
\begin{equation*}  
  |I_R(t) - I_R(t_0)| \leq C(t-t_0). 
\end{equation*}
By Claim \ref{R-claim}, we get that $I_R(t) \rightarrow 0$ as $t \to T^*$, for all $R > 0$. Hence
\begin{equation*} 
 I_R(t_0) \leq C(T^{*}-t_0) .
\end{equation*}
Since this bound is uniform in $R$, by taking $\ds R \rightarrow \infty$, we conclude
$u_c(t_0) \in L^2$, and so indeed
$\ds  u_c(t) \in L^2, t \in [0,T^{*})$.
Moreover defining
\[
  I(t) := \frac{1}{2} \int |u_c(t,x)|^2 dx,
\] 
we conclude that
\begin{equation} \label{L2vanish}
  I(t) \leq C(T^{*}-t).
\end{equation}
Now the $L^2$-dissipation relation~\eqref{L2disd} gives
\begin{equation*} I'(t) = - \int \left( |\nabla u_c|^2 - |u_c|^4 \right) dx=  -K(u_c(t)),
\end{equation*}
\[
  K(u) :=  \int \left( |\nabla u|^2 - |u|^4 \right) dx = 2 E(u) - \frac{1}{2} \int |u|^4 dx.
\] 
Now for any sequence $ \{t_n\}_n \nearrow T^{*}$, let (up to subsequence) 
\[
  \frac{1}{\lambda(t_k)} u_c(\frac{x-x_k}{\lambda(t_k)}, t_k) 
  \xrightarrow{\dot{H^1}} \bar{v} \in \dot{H}^1.
\]
Proceeding by contradiction, we suppose $\underline{E} \leq 0$. 
If so,
\[
  K(\bar{v}) = \lim_{k \to \infty} K(u_c(t_k)) = 2 \underline{E} - \frac{1}{2} \int |\bar{v}|^4 
  \leq - \frac{1}{2} \int |\bar{v}|^4 < 0,
\]
since $\bar{v} \equiv 0$ would contradict the assumption $T^* < \infty$. So
\[
  I'(t_k) = -K(u_c(t_k)) \to -K(\bar{v}) > 0.
\]
Thus $I'(t) > 0$ for all t sufficiently close to $T^{*};$ otherwise, we could find a subsequence along which $I' \leq 0,$ and the preceding argument would provide a contradiction. 
So $I(t)$ is increasing for $t$ near $T^*$, which contradicts~\eqref{L2vanish}.

Thus we have shown that $\underline{E} > 0$, 
completing the proof that $|x(t)|$ remains bounded.
\end{proof}

Since $|x(t)|$ remains bounded while $\lambda(t) \xrightarrow{t \rightarrow T^{*}} \infty,$ by the compactness we can find an $R_0 > 0$ large enough such that for all $x, |x| \geq R_0:$
\begin{equation*} 
  \|u_c \|_{L^{\infty}_t \dot{H}^1_x \cap L^{\infty}_t L^4_x (\Omega_{T^*} )} < \epsilon_0,
\end{equation*}
where $\Omega_{T^*}:= (0,T^*) \times B_{\sqrt{T^*}}(x_0)$.

By an appropriate scaling and shifting argument, the Regularity Lemma~\ref{regularity}
shows that $u_c$ is smooth on
$\Omega := (\mathbb{R}^4 \setminus B_{R_0}(0)) \times [\frac{3}{4}T^*,T^*]$, 
with uniform bounds on derivatives. 
Since $u$ is continuous up to $T^*$ outside $B_{R_0},$ Claim \ref{R-claim} implies that $u_c(x, T^*) \equiv 0, $ in the exterior of this ball. Since $u_c$ is bounded and smooth in $\Omega,$ an application of the Backwards Uniqueness Theorem~\ref{Backwards} implies that $u_c \equiv 0$ in
$\Omega.$ Define $\tilde{\Omega}:= \mathbb{R}^4 \times (\frac{3}{4}T^*,\frac{7}{8}T^*]$. 
Applying the Unique Continuation Theorem~\ref{unique_c} on a cylinder of sufficiently 
large spatial radius, centered at a point of $\Omega$, implies
$u_c \equiv 0$ in  $\tilde{\Omega}.$ By the uniqueness guaranteed by the local wellposedness theory we get that $u_c \equiv 0$, which contradicts (\ref{U-BU}).
\end{proof}

\section{Blow-up} \label{blowup section}

In this section we give criteria on the initial data 
which ensure that the corresponding solutions blow-up in finite-time.

The following result is well-known \cite{Lev, Ball, Caz} but we give the proof for the convenience of the reader.
\begin{proposition} \label{Levine} 
Solutions of
\begin{equation}
\begin{split}
  & \ds u_t = \Delta u + |u|^{p-1}u, \qquad  1 < p \leq 2^*-1 = \frac{d+2}{d-2} \\
  & u(x,0) = u_0(x) \in H^1(\R^d)
\end{split}
\label{bup-eq}
\end{equation} 
with
\[
  E(u_0) := \int_{\R^d} \left( \frac{1}{2} |\nabla u_0|^2 dx - 
  \frac{1}{p+1} |u_0|^{p+1} \right) dx < 0
\] 
must blow-up in finite-time, in the sense that there is no global
solution $u \in C([0, \infty); H^1(\R^d))$.
\end{proposition}
Notice that we can always find such initial data, e.g., if 
$u_0 = \lambda f, f \in H^1(\mathbb{R}^d)$ 
we can force negative energy by taking $\lambda$ large.

\begin{proof} 
We first derive some identities satisfied as long as a solution remains regular.

Multiplying the equation \eqref{bup-eq} first by $u$ and then by $u_t$ and integrating by parts we obtain the $L^2$-dissipation relation
\begin{equation}
\frac{d}{dt} \left(\frac{1}{2} \int_{\mathbb{R}^d} |u|^2 dx \right) = \int_{\mathbb{R}^d} |u|^{p+1} dx \label{bup-u} - \int_{\mathbb{R}^d} |\nabla u|^2 dx
=: - K(u)
\end{equation} 
and the energy dissipation relation
\begin{equation}
\int_{\mathbb{R}^d} |u_t|^2 dx = \frac{d}{dt} \left( \frac{1}{p+1}\int_{\mathbb{R}^d} |u|^{p+1} dx - \frac{1}{2} \int_{\mathbb{R}^d} |\nabla u|^2 dx \right)
= -\frac{d}{dt} E(u(t)).
\label{bup-u_t}
\end{equation} 
For convenience we define $J(t) := -E(t)$ and hence by (\ref{bup-u_t}) we have that
$\ds J'(t) := \int_{\mathbb{R}^d} |u_t|^2 dx  \geq 0$ and by the assumption on the energy $J(0) > 0.$ It will be also useful to write $J(t)$ as 
\begin{equation} J(t) = J(0) + \int_0^t \int_{\mathbb{R}^d} |u_t|^2 dx dt. \label{J}\end{equation}
Define 
\begin{equation} \label{Idef}
  I(t) = \int_0^t \int_{\mathbb{R}^d} |u|^2 dx dt + A 
\end{equation} 
with $A > 0, $ to be chosen later.
With this definition 
\begin{equation} 
  I'(t) = \int_{\mathbb{R}^d} |u|^2 dx  \label{I'(t)}
\end{equation} 
and
\begin{equation} 
 I''(t) = 2 \left(\int_{\mathbb{R}^d} |u|^{p+1} dx -  \int_{\mathbb{R}^d} |\nabla u|^2 dx \right). 
\end{equation}
Since $p>1, \delta:= \frac{1}{2} (p-1) > 0;$ a comparison with the energy functional yields
\begin{equation} 
I''(t) \geq 4 (1+\delta) J(t) = 4 (1+\delta) \left(J(0) + \int_0^t \int_{\mathbb{R}^d} |u_t|^2 dx dt \right ).\label{I''}
\end{equation} 
We can also rewrite
\[
  \ds I'(t) =  \int_{\mathbb{R}^d} |u|^2 dx  = \int_{\mathbb{R}^d} |u_0|^2 dx + 2 Re \int_0^t  
  \int_{\mathbb{R}^d} \bar{u} u_t dx dt.
\]
For any $\epsilon > 0$ the Young and H\"older inequalities give 
\begin{equation} 
(I'(t))^2 \leq 4 (1+\epsilon) \left(\int_0^t \int_{\mathbb{R}^d} |u|^2 dxdt \right)\left(\int_0^t \int_{\mathbb{R}^d} |u_t|^2 dxdt \right) + (1+\frac{1}{\epsilon}) \left(\int_{\mathbb{R}^d} |u_0|^2 dx \right)^2 \label{(I')^2} 
\end{equation}
Combining (\ref{I''}),(\ref{J}),(\ref{(I')^2}), for any $\alpha > 0$ we obtain:
\begin{equation} 
\begin{split}
  I''(t)I(t) -(1+\alpha) (I'(t))^2 &\geq 4 (1+\delta) \left[J(0) + \int_0^t \int_{\mathbb{R}^d} |u_t|^2 dx dt \right]\left[\int_0^t \int_{\mathbb{R}^d} |u|^2 dx dt + A  \right]  \\
\label{before DE}
  & \quad -4 (1+\epsilon) (1+\alpha) \left[\int_0^t \int_{\mathbb{R}^d} |u|^2 dxdt \right]\left[\int_0^t \int_{\mathbb{R}^d} |u_t|^2 dxdt \right] \\
  & \quad -(1+\frac{1}{\epsilon}) (1+\alpha) \left[\int_{\mathbb{R}^d} |u_0|^2 dx \right]^2 .
\end{split}
\end{equation}
Choose $\alpha,\epsilon$ small enough for $1+\delta \geq (1+\alpha) (1+\epsilon).$ Since $J(0)>0$ picking $A$ large enough we can ensure  $\ds I''(t)I(t) -(1+\alpha) (I'(t))^2 > 0.$ But this is equivalent to 
$\ds \frac{d}{dt} \left(\frac{I'(t)}{I^{\alpha+1}(t)}\right) > 0$ which in turn implies 
$\frac{I'(t)}{I^{\alpha+1}(t)} > \frac{I'(0)}{I^{\alpha+1}(0)} =:\tilde{a}$ for all $t > 0$.
Integrating  $I'(t) > \tilde{a} I^{\alpha+1}$ gives
\[
  \frac{1}{\alpha} \left(\frac{1}{I^{\alpha}(0)} - 
  \frac{1}{I^{\alpha}(t)} \right) > \tilde{a}t  \; \Rightarrow \; I^{\alpha}(t) > 
  \frac{I^{\alpha}(0)}{1-I^{\alpha}(0) \alpha\tilde{a}t} \rightarrow \infty
\] 
as $t \rightarrow \frac{1}{I^{\alpha}(0)\alpha\tilde{a}} = \frac{1}{A^{\alpha}\alpha\tilde{a}}=:\hat{t}$. 
This in turn implies that  $\ds \limsup_{t \to \hat{t}-} \|u\|_{L^2} = \infty$, showing that the solution cannot be globally in $C_t H^1$.
Note also that (\ref{bup-u}) implies $\ds \limsup_{t \to \hat{t}-} \|u\|_{L^{p+1}} = \infty$.
\end{proof}

We present a refinement in the \textit{critical case} which includes some positive energy data,
and in particular establishes Theorem~\ref{blowup crit1}.
So consider now equation~\eqref{CP_d}, for which
\[
  E(u) = \int_{\R^d} \left( \frac{1}{2} |\nabla u|^2 - \frac{1}{2*} |u|^{2^*} \right) dx.
\]
\begin{proposition} 
Let $u_0 \in H^1(\mathbb{R}^d)$ such that
\begin{equation} \label{BUset}
  E(u_0) < E(W) \hspace{0.5em} \text{and} \hspace{0.5em} 
  \|\nabla u_0\|_{L^2} \geq \|\nabla W\|_{L^2}.
\end{equation}
Then the corresponding solution $u$ to (\ref{CP_d}) blows up in finite time.
That is, $T_{max}(u_0)$ (coming from the $\dot H^1$ local theory as in 
Theorem~\eqref{lwp}) is finite.
\end{proposition}

\begin{proof}
We will give a sketch of the proof, which is largely a modification of the proof of the previous proposition.

By the Sobolev inequality~(\ref{Sobolev}), 
\begin{equation} 
   E(u) = \frac{1}{2}\int_{\R^d} |\nabla u|^2 dx - \frac{1}{2*} \int_{\R^d}  |u|^{2*} dx 
   \geq \frac{1}{2}\|\nabla u\|^2_{L^2} - \frac{1}{2*} \frac{\|W\|^{2*}_{L^{2*}}}
   {\|\nabla W\|^{2*}_{L^2}} \|\nabla u\|^{2*}_{L^2} .
\label{decreasing} 
\end{equation}
We define $\ds f(y):=\frac{1}{2}y - \frac{1}{2*} C^{2*} y^{\frac{2*}{2}}$, \;
$C = \frac{\| W \|_{L^{2*}}}{\| \nabla W \|_{L^2}} = \| \nabla W\|_{L^2}^{-{\frac{2}{d}}}$,
so that by energy dissipation and~\eqref{BUset},
\begin{equation}   \label{f(W)}
  f(\|\nabla u\|^2_{L^2}) \leq E(u) \leq E(u_0) < E(W).
\end{equation}
It is straightforward to verify that $f(y)$ is concave for $y \geq 0$ 
and attains its maximum value
$f(\|\nabla W\|^2_{L^2}) = E(W) = \frac{1}{d} \| \nabla W \|_{L^2}^2$ at 
$y = \|\nabla W\|^2_{L^2}$. 
Furthermore, it is strictly increasing on 
$[0, \| \nabla W \|_{L^2}^2]$ and strictly decreasing on 
$[\| \nabla W \|_{L^2}^2, +\infty)$. 
Denote the inverse function of $f$ on $[\| \nabla W \|_{L^2}^2, +\infty)$ as
\[
  e = f^{-1}  : (-\infty, E(W)] \rightarrow [\|\nabla W\|^{2}_{L^2}, +\infty),
\]
strictly decreasing.
By~\eqref{f(W)} and~\eqref{BUset} then,
\[
  \| \nabla u(t) \|_{L^2}^2 \geq e(E(u(t)).
\]
By the definitions of $K = K(u)$ and the energy $E = E(u)$
\[
\begin{split}
  -K(u) &= -\int_{\R^d} |\nabla u|^2 dx +\int_{\R^d}  |u|^{2*} dx =
  \frac{2}{d-2} \int_{\R^d} |\nabla u|^2 dx - 2^* E(u) \\ & \geq  
  \frac{2}{d-2} \left( e(E) -  d E \right) =: g(E).
\end{split}
\]
Note that $g(E(W)) = 0$ and for $E < E(W)$, $g(E) > 0$ and 
$g'(E) = \frac{2}{d-2} e'(E) - 2^* < -2^*$.
Defining $I(t)$ as in~\eqref{Idef}: 
\[
  I''(t) = -2K(u) \geq 2g(E(u)) > 0.
\]
By the Fundamental Theorem of Calculus and the energy dissipation relation,
\begin{equation*} 
  2 g(E(u)) = 2 g(E(u_0)) + 2 \int_0^t |g'(E(u(s)))| 
  \int_{\mathbb{R}^d} |u_t|^2 dx ds .
\end{equation*}
One can now repeat the proof of Proposition \ref{Levine} replacing (\ref{I''}) by 
\begin{equation}  
  I''(t) \geq 4 (1+\delta) J(t) = 4 (1+\delta) \left( 2g(E(u_0)) + \int_0^t 2 |g'(E(u(s)))| \int_{\mathbb{R}^d} |u_t|^2 dx ds \right ) \label{new I''} 
\end{equation}
Since $g(E(u_0)) > 0$, we can proceed exactly as in the proof of the
previous Proposition to conclude that if $T_{max} = \infty$,
then we must have $\ds \limsup_{t \to \hat{t}-} \| u(t) \|_{L^2} = \infty$
for some $\hat{t} < \infty$, which by~\eqref{I'(t)}
implies $\ds \limsup_{t \to \hat{t}-} \| u(t) \|_{L^{2*}} = \infty$, and so by
Sobolev, $\ds \limsup_{t \to \hat{t}-} \| \nabla u(t) \|_{L^2} = \infty$,
contradicting $T_{max} < \infty$.
\end{proof}


\section*{Acknowledgements}
The first author's research is supported through an NSERC Discovery Grant.


\end{document}